\documentclass[oneside,english]{amsart}
\usepackage[T1]{fontenc}
\usepackage[latin9]{inputenc}
\usepackage{array}
\usepackage{float}
\usepackage{enumitem}
\usepackage{amstext}
\usepackage{amsthm}
\usepackage{amssymb}
\usepackage{stmaryrd}
\usepackage{graphicx}

\makeatletter

\providecommand{\tabularnewline}{\\}

\numberwithin{equation}{section}
\numberwithin{figure}{section}

\usepackage{enumitem}
\usepackage{tikz}
\usepackage{graphicx}
\setlist[enumerate,1]{label=(\arabic*)}
\setlist[enumerate,2]{label=(\alph*)}
\setlist[enumerate,3]{label=(\roman*)}

\usepackage{etoolbox}

\AtBeginEnvironment{proof}{%
  \setlist[enumerate,1]{
    label=\textup{(\roman*)},
    ref=\textup{(\roman*)},
    leftmargin=*
  }%
}

\DeclareMathOperator{\con}{Con}
\DeclareMathOperator{\dom}{dom}
\DeclareMathOperator{\cg}{Cg}

\DeclareMathOperator{\rgic}{RGICon}
\DeclareMathOperator{\mic}{MICon}
\DeclareMathOperator{\cmic}{CMICon}
\DeclareMathOperator{\maxc}{maxCon}
\DeclareMathOperator{\sdw}{SDwidth}
\DeclareMathOperator{\cov}{cov}

\theoremstyle{definition}
\newtheorem{claimx}{Claim}
\newenvironment{customclaim}[1]
  {\renewcommand{\theclaimx}{#1}\begin{claimx}}
  {\end{claimx}}

\makeatother

\theoremstyle{plain}
\newtheorem{thm}{\protect\theoremname}
\newtheorem{lem}[thm]{\protect\lemmaname}
\theoremstyle{remark}
\newtheorem{rem}[thm]{\protect\remarkname}
\theoremstyle{plain}
\newtheorem{prop}[thm]{\protect\propositionname}
\newtheorem{cor}[thm]{\protect\corollaryname}
\usepackage{babel}
\providecommand{\corollaryname}{Corollary}
\providecommand{\lemmaname}{Lemma}
\providecommand{\propositionname}{Proposition}
\providecommand{\remarkname}{Remark}
\providecommand{\theoremname}{Theorem}

\begin{document}
\global\long\def\sub{\subseteq}%
\global\long\def\mins{\trianglelefteq}%
\global\long\def\S{\mathsf{S}}%
\global\long\def\Sup{\mathsf{S}^{\shortuparrow}}%
\global\long\def\T{\mathsf{T}}%
\global\long\def\te{\tau_{\mathrm{e}}}%
\global\long\def\tE{\tau_{\mathrm{E}}}%
\global\long\def\ted{\tau_{\mathrm{ed}}}%
\global\long\def\qrsi{\mathcal{Q}_{\mathrm{RSI}}}%
\global\long\def\nqrsi{\mathcal{Q}_{n\mathrm{RSI}}}%
\global\long\def\cp{\oplus}%

\title{Global representations of algebras with a near-unanimity term}
\author{Miguel Campercholi}
\begin{abstract}
Global representations are subdirect representations satisfying a
sheaf-like local-to-global patching principle. We develop a unified
and simplified theory of such representations. For quasivarieties
with a near-unanimity term, this framework recovers the main classical
representation theorems through a common argument and yields a general
representation by factors of bounded subdirect width. When the relatively
subdirectly irreducible members form a universal class, we obtain
an optimal result: every relatively congruence-distributive algebra
admits a global representation by relatively globally indecomposable
factors, which we characterize explicitly. We also provide a converse:
global representations by factors of bounded relative subdirect width
force the existence of a near-unanimity term.

Adding semisimplicity to the hypotheses of the main theorem yields
optimal representation results for filtral quasivarieties and dual
discriminator varieties, together with a simple description of the
globally indecomposable algebras.

The technical engine behind these results is a new infinitary extension
of the congruence-system component of the Baker--Pixley theorem.
\end{abstract}

\maketitle

\section{Introduction\label{sec:introduction}}

The purpose of this paper is to provide a unified and simplified theory
of global representations building on the congruence-system approach
that underlies much of the subject. For quasivarieties admitting a
near-unanimity term, the theory developed here both unifies and strengthens
the existing results. In particular, it sharpens Vaggione's general
theorem for arbitrary quasivarieties \cite{vaggione_2022} and yields
an optimal new representation theorem for relatively congruence-distributive
algebras in quasivarieties whose relatively subdirectly irreducible
members form a universal class up to trivial algebras.

Global subdirect products were introduced by Krauss and Clark \cite{krauss_clark_1979}
as a universal-algebraic counterpart of sheaf representations. Informally,
a subdirect product is global when elements defined on compatible
open pieces of the index space can be patched into a single element
of the represented algebra. Krauss and Clark showed that, for disjoint
factors, this is equivalent to the existence of a sheaf whose global
sections form a subdirect product of its stalks. Thus, global representations
refine Birkhoff's subdirect representations by supplementing point
separation with a local-to-global principle.

The power of global subdirect products comes precisely from this patching
property, which allows certain existential properties to be transferred
from the factors to the represented algebra. This is not generally
possible for ordinary subdirect products. For instance, every bounded
distributive lattice is a subdirect power of the two-element chain,
but complements need not transfer from the factors. In contrast, any
global subdirect product of complemented distributive lattices is
complemented, since complements can be constructed by patching.

This ability to preserve existential properties has made sheaf representations
a valuable tool in universal algebra, model theory, and logic. Most
classical applications rely on Boolean products, a highly structured
form of global subdirect product with particularly strong preservation
properties. Yet Boolean product representations are available only
in rather special settings. This motivates studying global subdirect
products in their bare-bones form: although they preserve fewer properties,
they retain the patching principle and arise much more broadly. Such
representations have led to significant applications of their own,
including the infinitary Baker--Pixley theorem for classes, Nachbin-type
results on congruence permutability, and descriptions of algebraic
functions \cite{Vaggione2018,GramagliaVaggione1996,GramagliaVaggione1997,CampercholiVaggione2011}.
The contrast is especially clear for finite families: every Boolean
product over a finite index set is a direct product, whereas a global
subdirect product over such a set may still be proper.

A central problem in the theory, formulated explicitly by Vaggione
\cite[Problem~17]{vaggione_2022}, is the following: given a quasivariety
$\mathcal{Q}$, find a class $\mathcal{F}$, as small and manageable
as possible and close to the relatively subdirectly irreducible members
of $\mathcal{Q}$, such that every algebra in $\mathcal{Q}$ admits
a global representation with factors in $\mathcal{F}$. In this situation,
we say that $\mathcal{F}$ globally represents $\mathcal{Q}$. The
natural optimal target is the class of relatively globally indecomposable
algebras, since its members admit no further nontrivial global decomposition.

An $(r+1)$-ary term $M$, with $r\geq2$, is a near-unanimity term
if
\[
M(x,\ldots,x,y,x,\ldots,x)\approx x
\]
for every possible position of $y$. If $\mathcal{Q}$ has such a
term, Vaggione proved that $\mathcal{Q}$ is globally represented
by the class of subalgebras of $r$-fold products of algebras in the
ultraproduct closure of its relatively subdirectly irreducible members
\cite[Theorem~19]{vaggione_2022}. The framework developed here yields,
as an almost immediate consequence, a strengthening of Vaggione's
theorem, in which all factors satisfying a natural criterion for global
decomposability are removed from the representing class (Theorem~\ref{thm:representacion global para NU}).
When $r=2$, this criterion simply says that the two projection kernels
permute.

Our main result, Theorem~\ref{thm:espectro global para RCD almost univ},
provides the optimal solution to the representation problem in a stronger,
algebra-by-algebra form. More precisely, if $\mathcal{Q}$ has a near-unanimity
term and its relatively subdirectly irreducible members form a universal
class up to trivial algebras, then every relatively congruence-distributive
algebra in $\mathcal{Q}$ admits a global representation with relatively
globally indecomposable factors. Consequently, the class of relatively
globally indecomposable algebras globally represents $\mathcal{Q}$
whenever $\mathcal{Q}$ is relatively congruence distributive. In
the varietal setting, no additional distributivity assumption is needed,
since a variety with a near-unanimity term is automatically congruence
distributive. We also give a concrete description of the relatively
globally indecomposable algebras involved. For $r=2$, this description
is especially simple: apart from the relatively subdirectly irreducible
members, they are exactly those whose relative congruence lattice
is atomic with exactly two nonpermuting atoms.

We determine rather precisely how far the hypotheses of the main theorem
can be weakened. Most importantly, the near-unanimity assumption is
necessary: the availability of global representations of the form
described above forces the quasivariety to have a near-unanimity term
(Theorem \ref{thm:global ancho finito sii NU}). Moreover, we construct
a counterexample showing that the condition on the relatively subdirectly
irreducible members cannot be replaced by the corresponding condition
on the relatively finitely subdirectly irreducible members.

When relative semisimplicity is added to the hypotheses of the main
theorem, one obtains the filtral setting \cite{CamRaf17-RelCongForm},
where the congruential characterization above can be sharpened to
a concrete description in terms of irredundant subdirect products
of simple algebras. This description becomes particularly explicit
for dual discriminator varieties: apart from the simple algebras,
their globally indecomposable members are precisely the simple crosses,
namely subalgebras of $\mathbf{S}\times\mathbf{S}'$ of the form
\[
(\{s\}\times S')\cup(S\times\{s'\}),
\]
where $\mathbf{S}$ and $\mathbf{S}'$ are simple. In discriminator
varieties, congruence permutability rules out these crosses, and the
classical Boolean representation theorem of Bulman-Fleming, Keimel
and Werner follows as an immediate consequence \cite{KeimelWerner1974,BulmanFlemingWerner1977,Werner1978}.

The main technical engine behind these results is an infinitary extension
of the congruence-system component of the Baker--Pixley theorem \cite[Theorem~2.1]{BakerPixley1975}.
It provides a compactness-based method for handling infinite systems
of congruences: for an algebra with an $(r+1)$-ary near-unanimity
term, the solvability of a suitably compact infinite system is determined
by its subsystems involving at most $r$ congruences. Besides powering
the representation results of the paper, this infinitary Baker--Pixley
theorem is of independent interest.

The paper is organized as follows. Section~\ref{sec:preliminaries}
fixes notation and collects the required compactness results. Section~\ref{sec:systems-of-congruences}
develops the theory of congruence systems and proves the infinitary
extension of the Baker--Pixley theorem. Section~\ref{sec:global-subdirect-products}
reviews global subdirect products and their connection with congruence
systems. Section~\ref{sec:general-global-representation} establishes
the general representation theorem for quasivarieties with a near-unanimity
term. Section~\ref{sec:finite-subdirect-width} studies relatively
globally indecomposable algebras of finite subdirect width, and Section~\ref{sec:almost-universal-rsi}
contains the main representation theorem and its converse, together
with the applications to congruence permutability and semisimple classes.
Finally, Section~\ref{sec:ado-semilattices} applies the main theorems
to ado-semilattices, a filtral variety of algebras of partial functions
under intersection and override, yielding an optimal global representation
result and a clean characterization of its globally indecomposable
members.

\section{Preliminaries\label{sec:preliminaries}}

In this section, we fix notation and review basic definitions and
facts needed throughout the paper.

An \emph{algebra} is a structure in a first-order language containing
only function and constant symbols. All classes of algebras considered
below consist of algebras in the same language. Throughout the paper,
boldface capital letters denote algebras (e.g., $\mathbf{A}$), and
the corresponding non-bold letters denote their universes (e.g., $A$).
Lowercase bold letters denote tuples. Unless otherwise stated, the
letters $k,\ell,m$, and $n$ denote non-negative integers, and $[m,n]$
denotes the interval $\{m,m+1,...,n\}$.

\medskip{}

For the remainder of this paper, $\mathcal{Q}$ denotes a given quasivariety,
and all relative notions are understood with respect to $\mathcal{Q}$.
Whenever a relative notion is applied to an algebra, that algebra
is tacitly assumed to belong to $\mathcal{Q}$.

\medskip{}

Given a class $\mathcal{K}$ of algebras, we write $\mathbb{I}(\mathcal{K})$,
$\mathbb{S}(\mathcal{K})$, and $\mathbb{P}_{u}(\mathcal{K})$ for
the classes of isomorphic copies, subalgebras, and ultraproducts of
members of $\mathcal{K}$, respectively. We denote by $\mathcal{K}^{+}$
the class obtained from $\mathcal{K}$ by adjoining all trivial algebras.

Let $\mathbf{A}$ be an algebra. By $\con(\mathbf{A})$ we denote
both the set and the lattice of congruences of $\mathbf{A}$. We write
$\Delta_{A}$ and $\nabla_{A}$ for the diagonal and total congruences
of $\mathbf{A}$, respectively, and omit the subscript $A$ when it
is clear from the context.

For a class $\mathcal{K}$ of algebras, define 
\[
\con_{\mathcal{K}}(\mathbf{A}):=\{\theta\in\con(\mathbf{A}):\mathbf{A}/\theta\in\mathbb{I}(\mathcal{K})\}.
\]

Recall that $\con_{\mathcal{Q}}(\mathbf{A})$, ordered by inclusion,
is an algebraic lattice whose arbitrary meets are intersections. We
denote its join operation by $\sqcup$. Its elements are called $\mathcal{Q}$\emph{-congruences},
or \emph{relative congruences}, of $\mathbf{A}$. For $F\subseteq A^{2}$,
let $\cg^{\mathbf{A}}_{\mathcal{Q}}(F)$ denote the least $\mathcal{Q}$-congruence
of $\mathbf{A}$ containing $F$. For $\Sigma\subseteq\con_{\mathcal{Q}}(\mathbf{A})$,
the up-set generated by $\Sigma$ is denoted by $[\Sigma)$. If $\Sigma=\{\theta\}$,
we may write $[\theta)$.

Let $\mic_{\mathcal{Q}}(\mathbf{A})$, $\cmic_{\mathcal{Q}}(\mathbf{A})$,
and $\maxc_{\mathcal{Q}}(\mathbf{A})$ denote, respectively, the sets
of meet-irreducible, completely meet-irreducible, and maximal proper
elements of $\con_{\mathcal{Q}}(\mathbf{A})$. By convention, $\nabla$
is not meet-irreducible.

A nontrivial algebra $\mathbf{A}\in\mathcal{Q}$ is \emph{relatively
finitely subdirectly irreducible} (RFSI),\emph{ relatively subdirectly
irreducible} (RSI), or \emph{relatively simple} (RS) if 
\[
\Delta\in\mic_{\mathcal{Q}}(\mathbf{A}),\qquad\Delta\in\cmic_{\mathcal{Q}}(\mathbf{A}),\qquad\text{or}\qquad\Delta\in\maxc_{\mathcal{Q}}(\mathbf{A}),
\]
respectively. We denote the corresponding classes by $\mathcal{Q}_{\mathrm{RFSI}}$,
$\mathcal{Q}_{\mathrm{RSI}}$, and $\mathcal{Q}_{\mathrm{RS}}$. The
quasivariety $\mathcal{Q}$ is relatively semisimple if 
\[
\mathcal{Q}_{\mathrm{RSI}}\subseteq\mathcal{Q}_{\mathrm{RS}}.
\]

An algebra $\mathbf{A}\in\mathcal{Q}$ is\emph{ relatively congruence-distributive}
(RCD) if $\con_{\mathcal{Q}}(\mathbf{A})$ is distributive. We shall
frequently use the fact that meet-irreducible elements of a distributive
lattice are meet-prime. Thus, if $\mathbf{A}$ is RCD, $\gamma\in\mic_{\mathcal{Q}}(\mathbf{A})$,
and $\Omega\subseteq\con_{\mathcal{Q}}(\mathbf{A})$ is finite, then
\[
\bigcap\Omega\subseteq\gamma\Longrightarrow\theta\subseteq\gamma\quad\text{for some }\theta\in\Omega.
\]
The quasivariety $\mathcal{Q}$ is relatively congruence-distributive
if each of its members is RCD.

We say that $\mathbf{A}$ is relatively congruence-permutable (RCP)
if 
\[
\theta\sqcup\delta=\theta\circ\delta
\]
for all $\theta,\delta\in\con_{\mathcal{Q}}(\mathbf{A})$. In this
case, the join operation of $\con_{\mathcal{Q}}(\mathbf{A})$ agrees
with that of $\con(\mathbf{A})$. An algebra that is both RCD and
RCP is called \emph{relatively arithmetic}.

A class $\mathcal{K}$ is \emph{universal} if it is axiomatizable
by universal first-order sentences. Recall that the least universal
class containing $\mathcal{S}$ is $\mathbb{ISP}_{u}(\mathcal{S})$.
Hence, $\mathcal{K}$ is universal if and only if 
\[
\mathbb{ISP}_{u}(\mathcal{K})=\mathcal{K}.
\]
We say that $\mathcal{K}$ is \emph{almost universal} if $\mathcal{K}^{+}$
is universal.

When a variety is under consideration, relative and absolute congruences
coincide; accordingly, we omit the adjective ``relative'' and the
prefix $\mathrm{R}$ from the corresponding terminology and notation.
For example, we write $\mathcal{V}_{\mathrm{SI}}$ instead of $\mathcal{V}_{\mathrm{RSI}}$.

An $(r+1)$-ary term $M(x_{1},\ldots,x_{r+1})$, where $r\geq2$,
is a \emph{near-unanimity term} (NU term) for a class $\mathcal{K}$
if
\[
\mathcal{K}\models M(y,x,\ldots,x)\approx M(x,y,x,\ldots,x)\approx\cdots\approx M(x,\ldots,x,y)\approx x.
\]
Whenever a near-unanimity term is under consideration, $r$ denotes
one less than its arity.

\subsection{Compactness}

The two topologies on the congruence lattice of an algebra introduced
below provide the compactness framework used throughout the paper
to pass from local congruence data to global solvability.

For $a,b\in A$ put 
\[
\mathrm{e}(a,b):=\{\theta\in\con(\mathbf{A}):\langle a,b\rangle\in\theta\}\qquad\mathrm{d}(a,b):=\{\theta\in\con(\mathbf{A}):\langle a,b\rangle\notin\theta\},
\]
and define $\tau_{\mathrm{e}}$ and $\tau_{\mathrm{ed}}$ as the topologies
on $\con(\mathbf{A})$ generated by $\{\mathrm{e}(a,b):a,b\in A\}$
and $\{\mathrm{e}(a,b):a,b\in A\}\cup\{\mathrm{d}(a,b):a,b\in A\}$,
respectively. Given $\Sigma\sub\con(\mathbf{A})$ we write $\overline{\Sigma}^{\mathrm{ed}}$
to denote its $\tau_{\mathrm{ed}}$-closure.
\begin{lem}
\label{lem:ted es booleana}The following hold:
\begin{enumerate}
\item The space $\langle\con(\mathbf{A}),\ted\rangle$ is a Boolean space.
In particular, it is compact and Hausdorff.
\item The set $\con_{\mathcal{Q}}(\mathbf{A})$ is $\ted$-closed in $\con(\mathbf{A})$.
Consequently, with the induced topology, $\con_{\mathcal{Q}}(\mathbf{A})$
is a Boolean space.
\item If $\Sigma$ is a $\ted$-closed subset of $\con_{\mathcal{Q}}(\mathbf{A})$,
then $\Sigma\setminus\{\nabla\}$ is $\te$-compact.
\item If $\Sigma$ is a $\te$-compact subset of $\con_{\mathcal{Q}}(\mathbf{A})$,
then $[\Sigma)$ is $\tau_{\mathrm{ed}}$-closed.
\end{enumerate}
\end{lem}

\begin{proof}
(1) and (2) follow from \cite[Lemma~5 and Corollary~7]{vaggione_2022},
respectively.

(3). If $\Sigma\sub\con_{\mathcal{Q}}(\mathbf{A})$ is $\ted$-closed,
then $\Sigma$ is $\ted$-compact by (2), and hence $\Sigma$ is $\te$-compact.
Now (3) follows from the fact that any $\te$ cover of $\Sigma\setminus\{\nabla\}$
also covers $\Sigma$.

(4). Suppose $\Sigma\sub\con_{\mathcal{Q}}(\mathbf{A})$ is $\te$-compact;
by item~(2), it suffices to show that $[\Sigma)$ is closed in $\con_{\mathcal{Q}}(\mathbf{A})$.
Fix $\theta\in\con_{\mathcal{Q}}(\mathbf{A})\setminus[\Sigma)$, and
for each $\delta\in\Sigma$ choose $p_{\delta}\in\delta\setminus\theta$.
As 
\[
\Sigma\sub\bigcup_{\delta\in\Sigma}\mathrm{e}(p_{\delta}),
\]
by compactness, there is a finite $\Sigma_{0}\sub\Sigma$ such that
\[
\Sigma\sub\bigcup_{\delta\in\Sigma_{0}}\mathrm{e}(p_{\delta}),
\]
and since equalizers are upsets
\[
[\Sigma)\sub\bigcup_{\delta\in\Sigma_{0}}\mathrm{e}(p_{\delta}).
\]
 Hence, $\bigcap_{\delta\in\Sigma_{0}}\mathrm{d}(p_{\delta})$ is
a $\ted$-open set containing $\theta$ and disjoint from $[\Sigma)$.
\end{proof}

\begin{lem}
\label{lem:compacidad 1}Let $\mathcal{S}$ be an almost universal
class. The following hold:
\begin{enumerate}
\item $\con_{\mathcal{S}}(\mathbf{A})\cup\{\nabla\}$ is $\tau_{\mathrm{ed}}$-closed.
\item $\con_{\mathcal{S}}(\mathbf{A})$ is $\te$-compact.
\item If $\Sigma\sub\con_{\mathcal{Q}}(\mathbf{A})$ is $\te$-compact,
then $[\Sigma)\cap\con_{\mathcal{S}}(\mathbf{A})$ is $\te$-compact.
\end{enumerate}
\end{lem}

\begin{proof}
(1). This follows from \cite[Cor. 6]{vaggione_2022}.

(2). Combine (1) and item (3) of Lemma~\ref{lem:ted es booleana}.

(3). Assume $\Sigma\sub\con_{\mathcal{Q}}(\mathbf{A})$ is $\te$-compact.
Item (4) of Lemma~\ref{lem:ted es booleana} says that $[\Sigma)$
is $\ted$-closed, and $\con_{\mathcal{S}}(\mathbf{A})\cup\{\nabla\}$
is $\ted$-closed by item (1). Hence, 
\[
[\Sigma)\cap(\con_{\mathcal{S}}(\mathbf{A})\cup\{\nabla\})=([\Sigma)\cap\con_{\mathcal{S}}(\mathbf{A}))\cup\{\nabla\}
\]
 is $\ted$-closed, and it follows from item (3) of Lemma~\ref{lem:ted es booleana}
that $[\Sigma)\cap\con_{\mathcal{S}}(\mathbf{A})$ is $\te$-compact.
\end{proof}

We shall also need two closure properties of $\tau_{\mathrm{ed}}$-closed
families.

Given $k>0$ and $\Gamma\sub\con_{\mathcal{Q}}(\mathbf{A})$ define
\[
\Gamma^{(k)}:=\{\bigcap X:X\sub\Gamma\text{ and }0<|X|\leq k\}.
\]

\begin{lem}
\label{lem:compacidad 2}Suppose $\Gamma\sub\con_{\mathcal{Q}}(\mathbf{A})$
is $\tau_{\mathrm{ed}}$-closed. The following hold:
\begin{enumerate}
\item $\Gamma^{(k)}$ is $\tau_{\mathrm{ed}}$-closed for every $k>0$.
\item $\Gamma$ is closed under $\sub$-directed unions.
\end{enumerate}
\end{lem}

\begin{proof}
(1). Since $\langle\con(\mathbf{A}),\ted\rangle$ is Hausdorff, it
suffices to prove that $\Gamma^{(k)}$ is $\ted$-compact. By Alexander's
lemma, we may only consider coverings by open sets from a sub-base
of $\tau_{\mathrm{ed}}$. So, take $R,S\sub A^{2}$ such that 
\[
\Gamma^{(k)}\sub\bigcup_{R}\mathrm{e}(p)\cup\bigcup_{S}\mathrm{d}(q).
\]
For $q\in S$ and $i\in[1,k]$ define 
\begin{align*}
V_{q,i} & :=\{\langle\theta_{1},\ldots,\theta_{k}\rangle\in\con_{\mathcal{Q}}(\mathbf{A})^{k}:q\notin\theta_{i}\};\\
V_{q} & :=\bigcup^{k}_{i=1}V_{q,i}.
\end{align*}
Note that each $V_{q}$ is an open set in the product space $\langle\con_{\mathcal{Q}}(\mathbf{A}),\tau_{\mathrm{ed}}\rangle^{k}$.
Furthermore, 
\[
\Gamma^{k}\sub\bigcup_{p\in R}\mathrm{e}(p)^{k}\cup\bigcup_{q\in S}V_{q}.
\]
By compactness there are $R_{0}\sub R$ and $S_{0}\sub S$, both finite,
such that 
\[
\Gamma^{k}\sub\bigcup_{p\in R_{0}}\mathrm{e}(p)^{k}\cup\bigcup_{q\in S_{0}}V_{q}.
\]
 Then, 
\[
\Gamma^{(k)}\sub\bigcup_{p\in R_{0}}\mathrm{e}(p)\cup\bigcup_{q\in S_{0}}\mathrm{d}(q),
\]
as desired.

(2). We start by recalling that $\con_{\mathcal{Q}}(\mathbf{A})$
is closed under $\sub$-directed unions. Let $D\sub\Gamma$ be directed
and put $\delta:=\bigcup D$. Aiming at a contradiction, suppose $\delta\notin\Gamma$.
Then $\Gamma=X\cup Y$ where $X:=\{\alpha\in\Gamma:\alpha\nsubseteq\delta\}$
and $Y:=\{\beta\in\Gamma:\delta\nsubseteq\beta\}$. For each $\alpha\in X$
take $p_{\alpha}\in\alpha\setminus\delta$, and for each $\beta\in Y$
pick $q_{\beta}\in\delta\setminus\beta$. By compactness there are
finite sets $X_{0}\sub X$ and $Y_{0}\sub Y$, such that 
\[
\Gamma\sub\bigcup_{\alpha\in X_{0}}\mathrm{e}(p_{\alpha})\cup\bigcup_{\beta\in Y_{0}}\mathrm{d}(q_{\beta}).
\]
Note that $D$ is also contained in this finite union, and, since
none of the $p_{\alpha}$'s is in $\delta$, we have 
\[
D\sub\bigcup_{\beta\in Y_{0}}\mathrm{d}(q_{\beta}).
\]
For each $\beta\in Y_{0}$ we have $q_{\beta}\in\delta$, so there
is $\gamma_{\beta}\in D$ with $q_{\beta}\in\gamma_{\beta}$. Then,
as $D$ is directed, there is $\gamma\in D$ such that $\{q_{\beta}:\beta\in Y_{0}\}\sub\gamma$.
But this is not possible, since $\gamma\in\bigcup_{\beta\in Y_{0}}\mathrm{d}(q_{\beta})$.
\end{proof}

Here is a useful fact about $\tau_{\mathrm{e}}$-compact sets.
\begin{lem}[{\cite[Lem. 8]{vaggione_2022}}]
\label{lem:completamente meet-prime}Suppose $\Gamma\sub\con_{\mathcal{Q}}(\mathbf{A})$
is $\tau_{\mathrm{e}}$-compact and $\gamma\in\con_{\mathcal{Q}}(\mathbf{A})$
is meet-prime, with $\bigcap\Gamma\sub\gamma$. Then there is $\delta\in\Gamma$
such that $\delta\sub\gamma$.
\end{lem}

\section{Systems of congruences\label{sec:systems-of-congruences}}

In this section we present our treatment of systems of congruence
equations. These constitute a key technical tool in the remainder
of the paper.

Let $\Sigma\subseteq\con(\mathbf{A})$. A \emph{system} on $\Sigma$
is a map $\mathsf{S}\colon\Sigma\rightarrow A$. Let $\S$ be a system
on $\Sigma$. If $\Sigma\subseteq\con_{\mathcal{Q}}(\mathbf{A})$
and 
\[
\langle\mathsf{S}(\theta),\mathsf{S}(\delta)\rangle\in\theta\sqcup\delta\qquad\text{for all }\theta,\delta\in\Sigma,
\]
we say that $\mathsf{S}$ is a $\sqcup$\emph{-system}.

A \emph{solution} for $\mathsf{S}$ is an element $a\in A$ such that
\[
\langle\mathsf{S}(\theta),a\rangle\in\theta\qquad\text{for all }\theta\in\Sigma.
\]
We say that $\mathsf{S}$ is \emph{solvable} if it has a solution.

For $u\subseteq\con(\mathbf{A})$, we denote by $\mathsf{S}|_{u}$
the restriction of $\mathsf{S}$ to $u\cap\Sigma$. We call $\mathsf{S}$
a $\tau_{\mathrm{e}}$\emph{-system} if there is $\mathcal{C}\subseteq\tau_{\mathrm{e}}$
such that $\mathcal{C}$ covers $\Sigma$ and $\mathsf{S}|_{u}$ is
solvable for every $u\in\mathcal{C}$. Such a cover $\mathcal{C}$
is said to be \emph{compatible} with $\mathsf{S}$.

Given $\Sigma,\Sigma^{+}\sub\con(\mathbf{A})$, we say that $\Sigma$
\emph{minorizes} $\Sigma^{+}$, in symbols $\Sigma\trianglelefteq\Sigma^{+}$,
if for every $\delta\in\Sigma^{+}$ there is $\theta\in\Sigma$ such
that $\theta\sub\delta$. If $\Sigma\trianglelefteq\Sigma^{+}$ but
$\Sigma^{+}\ntrianglelefteq\Sigma$, we say that $\Sigma$ \emph{strictly
minorizes} $\Sigma^{+}$, and write $\Sigma\vartriangleleft\Sigma^{+}$.

It is convenient to extend this notion to systems. Let $\mathsf{S}$
and $\mathsf{S}^{+}$ be systems on $\Sigma$ and $\Sigma^{+}$, respectively.
We say that $\mathsf{S}$ \emph{minorizes} $\mathsf{S}^{+}$, and
write $\mathsf{S}\mins\mathsf{S}^{+}$, provided that: 
\begin{itemize}
\item $\Sigma\trianglelefteq\Sigma^{+}$; 
\item for all $\theta\in\Sigma$ and $\delta\in\Sigma^{+}$ such that $\theta\sub\delta$,
we have $\langle\mathsf{S}(\theta),\mathsf{S}^{+}(\delta)\rangle\in\delta$.
\end{itemize}
\begin{lem}
\label{lem:minos}Let $\mathsf{S},\mathsf{S}^{+}$ be systems such
that $\S\mins\mathsf{S}^{+}$. Then, any solution for $\mathsf{S}$
is a solution for $\mathsf{S}^{+}$.
\end{lem}

\begin{proof}
Suppose $\S\mins\mathsf{S}^{+}$, and let $a\in A$ be a solution
for $\S$. Fix $\delta\in\dom\mathsf{S}^{+}$ and choose $\theta\in\dom\S$
with $\theta\sub\delta$. Then $\langle\S(\theta),\mathsf{S}^{+}(\delta)\rangle\in\delta$,
and since $\langle a,\S(\theta)\rangle\in\theta\sub\delta$, we have
$\langle a,\mathsf{S}^{+}(\delta)\rangle\in\delta$.
\end{proof}

Our next lemma shows that it is easy to extend a $\sqcup$-system
upwards.
\begin{lem}
\label{lem:extender sistema para arriba}Let $\mathsf{S}$ be a $\sqcup$-system
on $\Sigma\sub\con_{\mathcal{Q}}(\mathbf{A})$, and let $\Sigma^{+}\sub\con_{\mathcal{Q}}(\mathbf{A})$
be such that $\Sigma\mins\Sigma^{+}$. Then, there is a $\sqcup$-system
$\mathsf{S}^{+}$ on $\Sigma^{+}$ satisfying $\S\mins\mathsf{S}^{+}$.
Moreover, if $\Sigma\sub\Sigma^{+}$, we can choose $\mathsf{S}^{+}$
so that $\mathsf{S}^{+}|_{\Sigma}=\mathsf{S}$.
\end{lem}

\begin{proof}
For each $\theta\in\Sigma^{+}\setminus\Sigma$ choose $\delta_{\theta}\in\Sigma$
with $\delta_{\theta}\subseteq\theta$. Define 
\[
\mathsf{S}^{+}(\theta):=\begin{cases}
\mathsf{S}(\delta_{\theta}) & \text{if }\theta\in\Sigma^{+}\setminus\Sigma;\\
\mathsf{S}(\theta) & \text{if }\theta\in\Sigma.
\end{cases}
\]
It is straightforward to check that $\mathsf{S}^{+}$ satisfies the
desired properties.
\end{proof}

\begin{lem}
\label{lem:min por sistema finito implica te sistema}Let $\S$ be
a finite $\sqcup$-system and suppose $\S^{+}$ is a system such that
$\S\mins\S^{+}$. Then $\S^{+}$ is a $\tau_{\mathrm{e}}$-system.
In particular, every finite $\sqcup$-system is a $\tau_{\mathrm{e}}$-system.
\end{lem}

\begin{proof}
Suppose $\Omega$ is a finite subset of $\con_{\mathcal{Q}}(\mathbf{A})$.
Let $\mathsf{S}$ be a $\sqcup$-system on $\Omega$ and let $\S^{+}$
be a system satisfying $\S\mins\S^{+}$. By definition we have 
\[
\langle\mathsf{S}(\theta),\mathsf{S}(\theta')\rangle\in\theta\sqcup\theta'\qquad\text{for all }\theta,\theta'\in\Omega.
\]
Then, since $\con_{\mathcal{Q}}(\mathbf{A})$ is algebraic, for each
$\sigma\in\Omega$ there is a finite subset $F_{\sigma}\subseteq\sigma$
such that
\begin{equation}
\langle\mathsf{S}(\theta),\mathsf{S}(\theta')\rangle\in\cg^{\mathbf{A}}_{\mathcal{Q}}(F_{\theta})\sqcup\cg^{\mathbf{A}}_{\mathcal{Q}}(F_{\theta'})\qquad\text{for all }\theta,\theta'\in\Omega.\label{eq:sub finito}
\end{equation}
For $\sigma\in\Omega$ define 
\[
u_{\sigma}:=\bigcap_{p\in F_{\sigma}}\mathrm{e}(p).
\]
Note that $\mathcal{C}:=\{u_{\sigma}:\sigma\in\Omega\}$ is a $\tau_{\mathrm{e}}$-cover
of $\Omega$, since $\sigma\in u_{\sigma}$ for all $\sigma\in\Omega$.
We claim that $\mathcal{C}$ is compatible with $\mathsf{S}^{+}$.
First note that $\mathcal{C}$ covers $\dom\mathsf{S}^{+}$, because
$\Omega\mins\dom\mathsf{S}^{+}$ and each member of $\mathcal{C}$
is an up-set. Fix $\theta\in\Omega$; we show that $\mathsf{S}(\theta)$
is a solution for $\mathsf{S}^{+}|_{u_{\theta}}$. Let $\delta\in u_{\theta}\cap\dom\mathsf{S}^{+}$
and choose $\theta'\in\Omega$ with $\theta'\sub\delta$. Then, $F_{\theta}\cup F_{\theta'}\sub\delta$,
and it follows from (\ref{eq:sub finito}) that 
\[
\langle\mathsf{S}(\theta),\mathsf{S}(\theta')\rangle\in\delta.
\]
Now, as $\S\mins\S^{+}$ and $\theta'\sub\delta$, we have 
\[
\langle\mathsf{S}(\theta'),\mathsf{S}^{+}(\delta)\rangle\in\delta,
\]
therefore $\langle\mathsf{S}(\theta),\mathsf{S}^{+}(\delta)\rangle\in\delta$.
\end{proof}

We introduce two completeness properties that play a key role in this
article. A subset $\Sigma\sub\con(\mathbf{A})$ is \emph{$\tau_{\mathrm{e}}$-complete}
if every $\tau_{\mathrm{e}}$-system on $\Sigma$ is solvable.

If $\Sigma\sub\con_{\mathcal{Q}}(\mathbf{A})$, we say that $\Sigma$
is a \emph{Chinese remainder set} (CRS) if every $\sqcup$-system
on $\Sigma$ is solvable.

Note that if $\mathcal{R}$ is the variety of rings, and $\mathbf{Z}$
is the ring of integers, then the Chinese Remainder Theorem entails
that every finite subset of $\con(\mathcal{\mathbf{Z}})$ is a CRS.
This was generalized as follows by A.~Pixley in \cite{Pixley1972Completeness}.
\begin{lem}
\label{lem: relativamente aritmetica sii todo finito es CRS}The following
are equivalent:
\begin{enumerate}
\item $\mathbf{A}$ is relatively arithmetic.
\item Every finite subset of $\con_{\mathcal{Q}}(\mathbf{A})$ is a CRS.
\end{enumerate}
\end{lem}

\begin{proof}
Even though the result is stated and proved for absolute congruences
in \cite{Pixley1972Completeness}, the exact same proof works for
relative congruence lattices.
\end{proof}

Next we make precise the connection between finite CRSs and $\te$-complete
sets. Although it can be extracted from the arguments in~\cite[Lemma~5]{CampercholiVaggioneZigaran:RCDGlobal},
the terminology and formulation used here are slightly different,
so we include a proof.
\begin{lem}
\label{lem:CRS finito sii los joins son global spec}For a finite
$\Omega\subseteq\con_{\mathcal{Q}}(\mathbf{A})$ the following are
equivalent:
\begin{enumerate}
\item $\Omega$ is a CRS.
\item $\{\theta\sqcup\delta:\theta,\delta\in\Omega\}$ is $\te$-complete.
\end{enumerate}
\end{lem}

\begin{proof}
Put $\Sigma:=\{\theta\sqcup\delta:\theta,\delta\in\Omega\}$.

(1)$\Rightarrow$(2). Let $\S$ be a $\te$-system on $\Sigma$, and
suppose $\mathcal{C}\sub\te$ is a cover of $\Sigma$ compatible with
$\S$. For each $\theta\in\Omega$ choose $u_{\theta}\in\mathcal{C}$
such that $\theta\in u_{\theta}$, and fix a solution $a_{\theta}$
for $\S|_{u_{\theta}}$. We claim that $\T:\theta\mapsto a_{\theta}$
is a $\sqcup$-system on $\Omega$. Fix $\theta,\delta\in\Omega$
and note that, since $u_{\theta}$ and $u_{\delta}$ are up-sets,
we have $\theta\sqcup\delta\in u_{\theta}\cap u_{\delta}$. Now, as
$a_{\theta}$ and $a_{\delta}$ are solutions for the restrictions
of $\S$ to $u_{\theta}$ and $u_{\delta}$, respectively, it follows
that 
\[
\langle a_{\theta},\S(\theta\sqcup\delta)\rangle,\langle a_{\delta},\S(\theta\sqcup\delta)\rangle\in\theta\sqcup\delta.
\]
Hence, $\langle a_{\theta},a_{\delta}\rangle\in\theta\sqcup\delta$,
as desired. Since $\Omega$ is a CRS, there is a solution $a\in A$
for $\T$. We show that $a$ is a solution for $\S$. Fix $\theta\sqcup\delta\in\Sigma$
with $\theta,\delta\in\Omega$. As $a_{\theta}$ is a solution for
$\S|_{u_{\theta}}$ and $\theta\sqcup\delta\in u_{\theta}$, we have
\[
\langle a_{\theta},\S(\theta)\rangle\in\theta\text{ and }\langle a_{\theta},\S(\theta\sqcup\delta)\rangle\in\theta\sqcup\delta.
\]
 Also, since $a$ is a solution for $\T$, we obtain 
\[
\langle a,a_{\theta}\rangle\in\theta.
\]
Combining these facts we get $\langle a,\S(\theta\sqcup\delta)\rangle\in\theta\sqcup\delta$.

(2)$\Rightarrow$(1). Let $\S$ be a $\sqcup$-system on $\Omega$.
By Lemma \ref{lem:extender sistema para arriba}, there is a $\sqcup$-system
$\S^{+}$ on $\Sigma$ such that $\S^{+}|_{\Omega}=\mathsf{S}$. Since
$\S^{+}$ is finite, Lemma \ref{lem:min por sistema finito implica te sistema}
implies that $\S^{+}$ is a $\te$-system. Thus, $\S^{+}$ has a solution,
and therefore so does its restriction $\S$.
\end{proof}

We next recall the Correspondence Lemma for relative congruence lattices,
together with the corresponding facts for systems and

completeness.
\begin{lem}
\label{lem:correspondencia} Let $\theta\in\con_{\mathcal{Q}}(\mathbf{A})$. 
\begin{enumerate}
\item The map 
\[
\delta\mapsto\delta/\theta:=\{\langle a/\theta,b/\theta\rangle:\langle a,b\rangle\in\delta\}
\]
 is an isomorphism of bounded lattices from the sublattice of $\con_{\mathcal{Q}}(\mathbf{A})$
with universe $[\theta)$ onto $\con_{\mathcal{Q}}(\mathbf{A}/\theta)$.
\item Let $\Sigma\sub[\theta)$ and put $\Sigma/\theta:=\{\sigma/\theta:\sigma\in\Sigma\}$.
Let $\mathsf{S}$ be a system on $\Sigma$, and let $\mathsf{S}/\theta:\Sigma/\theta\rightarrow A/\theta$
be given by
\[
\mathsf{S}/\theta(\sigma/\theta):=\mathsf{S}(\sigma)/\theta.
\]

Then:
\begin{enumerate}
\item An element $a\in A$ is a solution for $\mathsf{S}$ if and only if
$a/\theta$ is a solution for $\mathsf{S}/\theta$.
\item The function $\mathsf{S}$ is a $\sqcup$-system on $\Sigma$ if and
only if $\mathsf{S}/\theta$ is a $\sqcup$-system on $\Sigma/\theta$.
Moreover, the same holds for $\te$-systems.
\item The set $\Sigma$ is a CRS if and only if $\Sigma/\theta$ is a CRS.
Moreover, the same holds for $\te$-complete sets.
\end{enumerate}
\end{enumerate}
\end{lem}

The next lemma allows us to pass to the $\tau_{\mathrm{ed}}$-closure
without losing $\tau_{\mathrm{e}}$-completeness.
\begin{lem}
\label{lem:global en la clausura}If $\Sigma\sub\con(\mathbf{A})$
is $\te$-complete, then the $\tau_{\mathrm{ed}}$-closure of $\Sigma$
is $\te$-complete.
\end{lem}

\begin{proof}
Suppose $\Sigma\sub\con(\mathbf{A})$ is $\te$-complete. Take a $\tau_{\mathrm{e}}$-system
$\mathsf{S}$ on $\overline{\Sigma}^{\mathrm{ed}}$, and let $\mathcal{C}\sub\tau_{\mathrm{e}}$
be a cover of $\overline{\Sigma}^{\mathrm{ed}}$ compatible with $\mathsf{S}$.
Note that $\mathcal{C}$ is compatible with $\mathsf{S}|_{\Sigma}$.
Since $\Sigma$ is $\te$-complete, there is $a\in A$ satisfying
\begin{equation}
\langle a,\mathsf{S}(\theta)\rangle\in\theta\qquad\text{for all }\theta\in\Sigma.\label{eq:a sol en Sigma}
\end{equation}
We claim that $a$ is a solution for $\mathsf{S}$. Indeed, fix $\delta\in\overline{\Sigma}^{\mathrm{ed}}$
and take $u\in\mathcal{C}$ such that $\delta\in u$. Since $\mathcal{C}$
is compatible with $\mathsf{S}$, there exists $b\in A$ with
\begin{equation}
\langle b,\mathsf{S}(\theta)\rangle\in\theta\qquad\text{for all }\theta\in u\cap\overline{\Sigma}^{\mathrm{ed}}.\label{eq:b sol en u}
\end{equation}
It follows from (\ref{eq:a sol en Sigma}) and (\ref{eq:b sol en u})
that
\[
u\cap\Sigma\sub\mathrm{e}(a,b),
\]
and hence
\[
u\cap\overline{\Sigma}^{\mathrm{ed}}\sub\overline{\mathrm{e}(a,b)}^{\mathrm{ed}}=\mathrm{e}(a,b).
\]
So $\langle a,b\rangle\in\delta$, which in view of (\ref{eq:b sol en u})
yields $\langle a,\mathsf{S}(\delta)\rangle\in\delta$, as desired.

The next notion identifies families for which system-solvability can
be tested on the completely meet-irreducible congruences lying above
them.
\end{proof}

We say that a subset $\Sigma\sub\con_{\mathcal{Q}}(\mathbf{A})$ is
\emph{CMI-saturated} if $[\Sigma)\cap\cmic_{\mathcal{Q}}(\mathbf{A})\sub\Sigma$.
\begin{lem}
\label{lem:CMI alcanza}If $\Sigma\sub\con_{\mathcal{Q}}(\mathbf{A})$
is CMI-saturated, then:
\begin{enumerate}
\item Every $\te$-system on $\Sigma$ is a $\sqcup$-system.
\item If $\S$ is a $\sqcup$-system on $\Sigma$ such that $\S|_{[\Sigma)\cap\cmic_{\mathcal{Q}}(\mathbf{A})}$
is solvable, then $\S$ is solvable.
\end{enumerate}
\end{lem}

\begin{proof}
(1) Let $\S$ be a $\te$-system on $\Sigma$. Fix $\theta,\delta\in\Sigma$
and $\gamma\in\cmic_{\mathcal{Q}}(\mathbf{A})$ such that $\theta,\delta\sub\gamma$.
Note that $\gamma\in\Sigma$ since $\Sigma$ is CMI-saturated. We
show that $\langle\S(\theta),\S(\delta)\rangle\in\gamma$. Let $\mathcal{C}\sub\tau_{\mathrm{e}}$
be a cover of $\Sigma$ compatible with $\mathsf{S}$, and choose
$u,v\in\mathcal{C}$ such that $\theta\in u$ and $\delta\in v$.
Let $a_{u},a_{v}\in A$ be solutions for $\S|_{u}$ and $\S|_{v}$,
respectively. Since $u$ and $v$ are increasing sets, it follows
that $\gamma\in u\cap v$, and hence 
\[
\langle a_{u},\S(\gamma)\rangle,\langle a_{v},\S(\gamma)\rangle\in\gamma.
\]
So, as $\langle a_{u},\S(\theta)\rangle$ and $\langle a_{v},\S(\delta)\rangle$
both belong to $\gamma$, we have that $\langle\S(\theta),\S(\delta)\rangle\in\gamma$.
Finally, since $\gamma$ is an arbitrary congruence in $\cmic_{\mathcal{Q}}(\mathbf{A})$
containing $\theta\sqcup\delta$, we conclude that $\langle\S(\theta),\S(\delta)\rangle\in\theta\sqcup\delta$.

(2) Let $a\in A$ be a solution for $\S|_{[\Sigma)\cap\cmic_{\mathcal{Q}}(\mathbf{A})}$
and fix $\theta\in\Sigma$. If $\gamma\in[\theta)\cap\cmic_{\mathcal{Q}}(\mathbf{A})$,
then $\langle\S(\gamma),\S(\theta)\rangle\in\gamma\sqcup\theta=\gamma$.
Also, as $\langle a,\S(\gamma)\rangle\in\gamma$, we have $\langle a,\S(\theta)\rangle\in\gamma$.
Thus,
\[
\langle a,\S(\theta)\rangle\in\bigcap([\theta)\cap\cmic_{\mathcal{Q}}(\mathbf{A}))=\theta.
\]
\end{proof}

\subsection{The Baker-Pixley theorem for infinite congruence systems}

In their classical paper, Baker and Pixley showed that the existence
of a near-unanimity term yields, among other consequences, a term-interpolation
result for functions with finite domain and a local-to-global solvability
result for finite congruence systems \cite{BakerPixley1975}. Infinitary
versions of the interpolation result were subsequently obtained in
\cite{Vaggione2018,CampercholiVaggione2023}. We now establish an
infinitary counterpart of the congruence-system result.

Let $\Sigma\sub\con(\mathbf{A})$ and $\mathsf{S}$ be a system on
$\Sigma$. For $r>0$, we say that $\mathsf{S}$ is $r$\emph{-solvable}
if $\mathsf{S}|_{\Omega}$ is solvable for each $\Omega\sub\Sigma$
with $|\Omega|\leq r$.
\begin{thm}
\label{thm:BP congruencial}Suppose $\mathbf{A}$ has a $(r+1)$-ary
NU term, and let $\Sigma\sub\con(\mathbf{A})$ be $\tau_{\mathrm{e}}$-compact.
If $\mathsf{S}$ is a $r$-solvable $\te$-system on $\Sigma$, then
$\mathsf{S}$ is solvable.
\end{thm}

\begin{proof}
Let $\mathcal{C}\sub\tau_{\mathrm{e}}$ be a cover of $\Sigma$ compatible
with $\mathsf{S}$. For each $u\in\mathcal{C}$ choose a solution
$a_{u}$ for $\mathsf{S}|_{u}$. Given $\mathbf{u}=\langle u_{1},\ldots,u_{r}\rangle\in\mathcal{C}^{r}$
and $c\in A$ let 
\[
V_{\mathbf{u},c}:=(u_{1}\cap\mathrm{e}(a_{u_{1}},c))\times\dots\times(u_{r}\cap\mathrm{e}(a_{u_{r}},c)).
\]
Note that members of $V_{\mathbf{u},c}$ are exactly those $\langle\theta_{1},\ldots,\theta_{r}\rangle\in u_{1}\times\dots\times u_{r}$
for which $c$ solves $\S|_{\{\theta_{1},\ldots,\theta_{r}\}}$. Also,
note that each of these sets is open in the topological space $\langle\con(\mathbf{A}),\tau_{\mathrm{e}}\rangle^{r}$.
Moreover, since $\mathcal{C}$ covers $\Sigma$ and $\mathsf{S}$
is $r$-solvable, we have 
\[
\Sigma^{r}\sub\bigcup\{V_{\mathbf{u},c}:\langle\mathbf{u},c\rangle\in\mathcal{C}^{r}\times A\}.
\]
As $\Sigma$ is $\tau_{\mathrm{e}}$-compact, we have that $\Sigma^{r}$
is compact in $\langle\con(\mathbf{A}),\tau_{\mathrm{e}}\rangle^{r}$,
and hence there is a finite subset $P\sub\mathcal{C}^{r}\times A$
with 
\begin{equation}
\Sigma^{r}\sub\bigcup\{V_{\mathbf{u},c}:\langle\mathbf{u},c\rangle\in P\}.\label{eq:P cubre}
\end{equation}
Let $C\sub A$ be the set of second coordinates of pairs in $P$.
It follows from (\ref{eq:P cubre}) that: 
\begin{equation}
\text{For each }\Omega\sub\Sigma\text{ with }|\Omega|\leq r\text{ there is a solution }c\in C\text{ for }\S|_{\Omega}.\label{eq: S r-interpola}
\end{equation}
For each $\theta\in\Sigma$ define
\[
C_{\theta}:=\{c\in C:\langle\mathsf{S}(\theta),c\rangle\in\theta\}.
\]
For $X\sub C$ put 
\[
\Sigma_{X}:=\{\theta\in\Sigma:C_{\theta}=X\},
\]
and set
\[
\mathcal{X}:=\{X\sub C:\Sigma_{X}\neq\emptyset\}.
\]
We prove by induction on $n$ that for any $X_{1},\ldots,X_{n}\in\mathcal{X}$
there exists $a\in A$ such that
\[
\langle\mathsf{S}(\theta),a\rangle\in\theta\quad\text{whenever }\;C_{\theta}\in\{X_{1},\ldots,X_{n}\}.
\]
Suppose first that $n\leq r$. For each $k\in[1,n]$ pick $\theta_{k}\in\Sigma_{X_{k}}$
and note that by (\ref{eq: S r-interpola}) there is $c\in C$ with
$\langle\mathsf{S}(\theta_{k}),c\rangle\in\theta_{k}$ for all $k$.
Thus $a:=c$ works in this case. Suppose next that $n>r$, and fix
$X_{1},\ldots,X_{n}\in\mathcal{X}$. By our inductive hypothesis,
for each $k\in[1,r+1]$ there is $a_{k}\in A$ satisfying 
\[
\langle\mathsf{S}(\theta),a_{k}\rangle\in\theta\text{ for all }\theta\text{ with }C_{\theta}\in\{X_{1},\ldots,X_{n}\}\setminus\{X_{k}\}.
\]
Let $M$ be a $(r+1)$-ary NU term for $\mathbf{A}$. We claim that
$a:=M^{\mathbf{A}}(a_{1},\ldots,a_{r+1})$ is the desired element.
Indeed, take $\theta\in\Sigma$ with $C_{\theta}=X_{\ell}$ for some
$\ell\in[1,n]$. For each $k\in[1,r+1]\setminus\{\ell\}$ we have
$\langle\mathsf{S}(\theta),a_{k}\rangle\in\theta$. Therefore, 
\[
a/\theta=M^{\mathbf{A}/\theta}(a_{1}/\theta,\ldots,a_{r+1}/\theta)=\mathsf{S}(\theta)/\theta,
\]
since at least $r$ of $a_{1}/\theta,\ldots,a_{r+1}/\theta$ agree
with $\mathsf{S}(\theta)/\theta$. 

To conclude, observe that, as $\mathcal{X}$ is finite, there is $a\in A$
with $\langle\mathsf{S}(\theta),a\rangle\in\theta$ for all $\theta\in\Sigma$.
\end{proof}

Since every $r$-solvable system with $r\geq2$ is a $\sqcup$-system,
Lemma~\ref{lem:min por sistema finito implica te sistema} allows
us to recover the original Baker--Pixley theorem from Theorem~\ref{thm:BP congruencial}.

\section{Global subdirect products\label{sec:global-subdirect-products}}

In this section we recall the definition of global subdirect products
and show how this notion can be recast in terms of congruence systems.

Let $\langle\mathbf{A}_{i}:i\in I\rangle$ be a family of algebras.
For $a,b\in\prod_{I}A_{i}$ put 
\[
\mathrm{E}(a,b):=\{i\in I:a(i)=b(i)\}.
\]
Fix a subdirect product $\mathbf{A}\leq\prod_{i\in I}\mathbf{A}_{i}$.
We introduce terminology relative to this representation, omitting
explicit reference to it. For each $i\in I$ let $\pi^{A}_{i}:A\to A_{i}$
denote the canonical projection.

Let $\tau$ be a topology on $I$. A \emph{$\tau$-patching system}
is a family 
\[
\langle a_{U}:U\in\mathcal{C}\rangle,
\]
 with $\mathcal{C}\subseteq\tau$ an open cover of $I$ and $\{a_{U}:U\in\mathcal{C}\}\sub A$,
satisfying 
\[
U\cap V\subseteq\mathrm{E}(a_{U},a_{V})\quad\text{for all }U,V\in\mathcal{C}.
\]
A \emph{solution} to such a patching system is an element $a\in A$
satisfying 
\[
U\subseteq\mathrm{E}(a,a_{U})\quad\text{for all }U\in\mathcal{C}.
\]
We say that $A$ \emph{patches over} $\tau$ if every patching system
over $\tau$ has a solution. 

We can now introduce the key definitions of this section:
\begin{itemize}
\item The subdirect product $\mathbf{A}\leq\prod_{i\in I}\mathbf{A}_{i}$
is \emph{global} if it patches over the topology generated by $\{\mathrm{E}(a,b):a,b\in A\}$.
(This topology is denoted by $\tau_{\mathrm{E}}$.) 
\item A \emph{global representation} of $\mathbf{A}$ is an embedding $\gamma:\mathbf{A}\hookrightarrow\prod_{i\in I}\mathbf{A}_{i}$
whose image is a global subdirect product.
\item A set $\Sigma\sub\con(\mathbf{A})$ is a \emph{global spectrum} of
$\mathbf{A}$ if the natural map $\mathbf{A}\rightarrow\prod_{\theta\in\Sigma}\mathbf{A}/\theta$
is a global representation.
\end{itemize}
Note that if $\Sigma$ is a global spectrum of $\mathbf{A}$, then
$\bigcap\Sigma=\Delta$.

The following lemma translates the patching property of a subdirect
product into the solvability of congruence systems. Since the verification
is routine, we include only a proof sketch.
\begin{lem}
\label{lem:prod subd es global sii el espectro es global}Let $\mathbf{A}\leq\prod_{I}\mathbf{A}_{i}$
be a subdirect product. The following are equivalent:
\begin{enumerate}
\item The subdirect product $\mathbf{A}\leq\prod_{I}\mathbf{A}_{i}$ is
global.
\item The set $\{\ker\pi^{A}_{i}:i\in I\}$ is $\te$-complete.
\end{enumerate}
\end{lem}

\begin{proof}[Proof sketch]
Put 
\[
\eta_{i}:=\ker\pi^{\mathbf{A}}_{i},\qquad\Sigma:=\{\eta_{i}:i\in I\},
\]
and let $p\colon I\to\Sigma$ be defined by $p(i)=\eta_{i}$. The
key observation is that, for every $a,b\in A$, 
\[
p^{-1}\bigl(\mathrm{e}(a,b)\cap\Sigma\bigr)=\mathrm{E}(a,b).
\]
Hence $p$ identifies the topology induced by $\tau_{\mathrm{e}}$
on $\Sigma$ with $\tau_{\mathrm{E}}$: a set $u\subseteq\Sigma$
corresponds to $p^{-1}(u)\subseteq I$. Under this correspondence,
compatible local solutions of a system on $\Sigma$ correspond precisely
to patching systems, since 
\[
i\in\mathrm{E}(a,b)\Longleftrightarrow\langle a,b\rangle\in\eta_{i}.
\]
Moreover, in both settings the condition for $a\in A$ to be a global
solution is the same. Possible repetitions among the congruences $\eta_{i}$
cause no difficulty, since every $\tau_{\mathrm{E}}$-open set is
a union of fibers of $p$. The equivalence now follows by translating
systems and their solutions along $p$. 
\end{proof}

\begin{lem}
\label{lem:global iff te-complete}For $\Sigma\sub\con(\mathbf{A})$
the following are equivalent:
\begin{enumerate}
\item $\Sigma$ is a global spectrum of $\mathbf{A}$.
\item $\Sigma$ is $\te$-complete and $\bigcap\Sigma=\Delta$.
\end{enumerate}
\end{lem}

\begin{proof}
Let 
\[
\eta_{\Sigma}\colon\mathbf{A}\longrightarrow\prod_{\theta\in\Sigma}\mathbf{A}/\theta
\]
be the natural map. This map is injective if and only if $\bigcap\Sigma=\Delta$.
When this condition holds, we identify $\mathbf{A}$ with $\eta_{\Sigma}(\mathbf{A})$;
under this identification, the kernels of the coordinate projections
are precisely the members of $\Sigma$. Hence, by Lemma~\ref{lem:prod subd es global sii el espectro es global},
$\eta_{\Sigma}(\mathbf{A})$ is global if and only if $\Sigma$ is
$\te$-complete. The result follows.
\end{proof}

It is worth noting that if $\Sigma$ is $\te$-complete and $\bigcap\Sigma\neq\Delta$,
it still induces a global representation of the quotient $\mathbf{A}/\bigcap\Sigma$.

The algebra $\mathbf{A}$ is \textit{relatively globally indecomposable}
(RGI) provided that $\mathbf{A}\in\mathcal{Q}$ and for every global
representation $\gamma:\mathbf{A}\hookrightarrow\prod_{I}\mathbf{A}_{i}$,
with each $\mathbf{A}_{i}\in\mathcal{Q}$, there exists $i\in I$
such that $\pi_{i}\circ\gamma$ is injective. We write $\mathcal{Q}_{\mathrm{RGI}}$
to denote the class of all RGI algebras in $\mathcal{Q}$.

The following congruential characterization of globally indecomposables
follows at once from Lemmas \ref{lem:global iff te-complete} and
\ref{lem:prod subd es global sii el espectro es global}.
\begin{lem}
\label{lem:carac congruencial de RGI} The following are equivalent:
\begin{enumerate}
\item The algebra $\mathbf{F}$ is relatively globally indecomposable.
\item If $\Sigma\sub\con_{\mathcal{Q}}(\mathbf{F})$ is $\te$-complete
and $\bigcap\Sigma=\Delta$, then $\Delta\in\Sigma$.
\end{enumerate}
\end{lem}

\begin{rem}
\label{rem:CRS implica descomponible}Note that Lemmas \ref{lem:CRS finito sii los joins son global spec}
and \ref{lem:carac congruencial de RGI}, in combination, imply that
if there is a finite CRS $\Sigma\sub\con_{\mathcal{Q}}(\mathbf{A})$
such that $\bigcap\Sigma\notin\Sigma$, then $\mathbf{A}/\bigcap\Sigma$
has nontrivial global decomposition.
\end{rem}

A relative congruence $\theta\in\con_{\mathcal{Q}}(\mathbf{A})$ is
an \emph{RGI-congruence} of $\mathbf{A}$ if $\mathbf{A}/\theta\in\mathcal{Q}_{\mathrm{RGI}}$.
We write $\rgic_{\mathcal{Q}}(\mathbf{A})$ for the set of all RGI-congruences
of $\mathbf{A}$. Note that $\cmic_{\mathcal{Q}}(\mathbf{A})\sub\rgic_{\mathcal{Q}}(\mathbf{A})$.

An algebra $\mathbf{A}\in\mathcal{Q}$ is \emph{RGI-spectral} if $\rgic_{\mathcal{Q}}(\mathbf{A})$
is a global spectrum of $\mathbf{A}$. The quasivariety $\mathcal{Q}$
is \emph{RGI-spectral} if every algebra in $\mathcal{Q}$ is RGI-spectral.

\subsection{Boolean products}

A subdirect product 
\[
\mathbf{A}\leq\prod_{i\in I}\mathbf{A}_{i}
\]
is a \emph{Boolean product} if there is a Boolean space topology $\tau$
on $I$ satisfying the following conditions:
\begin{itemize}
\item $\mathrm{E(}a,b)$ is clopen for all $a,b\in A$; 
\item for all $a,b,c_{0},c_{1}\in A$, the patching system which assigns
$c_{0}$ to $\mathrm{E(}a,b)$ and $c_{1}$ to its complement has
a solution.
\end{itemize}
See \cite{BurrisSankappanavar2012} for a thorough introduction to
Boolean products.

A \emph{Boolean representation} of an algebra $\mathbf{B}$ is an
embedding $\gamma:\mathbf{B}\hookrightarrow\prod_{i\in I}\mathbf{B}_{i}$
whose image is a Boolean product. We say that a subset $\Sigma\sub\con(\mathbf{A})$
is a \emph{Boolean spectrum} of $\mathbf{A}$ if the natural map $\mathbf{A}\rightarrow\prod_{\theta\in\Sigma}\mathbf{A}/\theta$
is a Boolean representation of $\mathbf{A}$. Note that this requires
$\bigcap\Sigma=\Delta$.
\begin{lem}
\label{lem:boolean spectra}Let $\mathbf{A}\in\mathcal{Q}$ and suppose
$\Sigma\sub\con(\mathbf{A})$ satisfies:
\begin{itemize}
\item $\bigcap\Sigma=\Delta$;
\item $\Sigma$ is $\ted$-closed;
\item for all $a,b,c_{0},c_{1}\in A$ the system $\S:\Sigma\rightarrow A$
given by 
\[
\S(\theta):=\begin{cases}
c_{0} & \text{if }\theta\in\mathrm{e}(a,b),\\
c_{1} & \text{otherwise},
\end{cases}
\]
has a solution.
\end{itemize}
Then, $\Sigma$ is a Boolean spectrum of $\mathbf{A}$.

\end{lem}

\begin{proof}
Since $\Sigma$ is $\ted$-closed, item~(1) of Lemma~\ref{lem:ted es booleana}
implies that $\Sigma$, endowed with the induced topology, is a Boolean
space. Furthermore, since the equalizers of the natural representation
are the sets $\mathrm{e}(a,b)\cap\Sigma$, the three assumptions say
precisely that this representation is a Boolean product.
\end{proof}

Under additional assumptions, a global spectrum yields a Boolean spectrum.
\begin{prop}
\label{prop:global incomparable es booleano}Let $\Sigma\sub\con(\mathbf{A})$
be a global spectrum of $\mathbf{A}$ such that $\Sigma\cup\{\nabla\}$
is $\ted$-closed and the congruences in $\Sigma$ are pairwise incomparable.
Then $\Sigma\cup\{\nabla\}$ is a Boolean spectrum of $\mathbf{A}$.
\end{prop}

\begin{proof}
Let $\Sigma$ be as in the statement and put $\Sigma^{+}:=\Sigma\cup\{\nabla\}$.
Observe that, as $\Sigma^{+}$ is $\ted$-closed, it is $\ted$-compact.
We apply Lemma~\ref{lem:boolean spectra} to prove that $\Sigma^{+}$
is a Boolean spectrum of $\mathbf{A}$. Fix $a,b,c_{0},c_{1}\in A$,
and consider the system $\S\colon\Sigma^{+}\rightarrow A$ given by
\[
\S(\theta):=\begin{cases}
c_{0} & \text{if }\theta\in\mathrm{e}(a,b),\\
c_{1} & \text{otherwise.}
\end{cases}
\]
We show that $\S|_{\Sigma}$ is a $\te$-system and hence solvable.

\noindent\begin{customclaim}{I}For each $\delta\in\mathrm{d}(a,b)\cap\mathrm{d}(c_{0},c_{1})\cap\Sigma$
there is a $\te$-open set $u_{\delta}$ such that $\delta\in u_{\delta}$
and 
\[
\Sigma\cap u_{\delta}\cap\mathrm{d}(c_{0},c_{1})\sub\mathrm{d}(a,b).
\]
\end{customclaim}

\noindent Suppose $\delta\in\mathrm{d}(a,b)\cap\mathrm{d}(c_{0},c_{1})\cap\Sigma$;
we claim that
\[
\Sigma^{+}\cap\bigcap_{p\in\delta}\mathrm{e}(p)\cap\mathrm{d}(c_{0},c_{1})\cap\mathrm{e}(a,b)=\emptyset.
\]
Indeed, if $\theta$ belongs to the intersection on the left-hand
side, then $\theta\neq\nabla$, since $\langle c_{0},c_{1}\rangle\notin\theta$.
Also, as $\theta\in\bigcap_{p\in\delta}\mathrm{e}(p)$, we have $\delta\sub\theta$.
But $\langle a,b\rangle\in\theta\setminus\delta$ , and hence $\delta\subsetneq\theta$,
which contradicts the fact that the congruences in $\Sigma$ are pairwise
incomparable. Thus,
\[
\Sigma^{+}\sub\bigcup_{p\in\delta}\mathrm{d}(p)\cup\mathrm{d}(c_{0},c_{1})\cup\mathrm{e}(a,b).
\]
By the $\ted$-compactness of $\Sigma^{+}$, there is a finite $F\sub\delta$
such that
\[
\Sigma^{+}\sub\bigcup_{p\in F}\mathrm{d}(p)\cup\mathrm{d}(c_{0},c_{1})\cup\mathrm{e}(a,b).
\]
Hence, taking $u_{\delta}:=\bigcap_{p\in F}\mathrm{e}(p)$ we have
\[
\Sigma^{+}\cap u_{\delta}\cap\mathrm{d}(c_{0},c_{1})\cap\mathrm{e}(a,b)=\emptyset,
\]
which entails $\Sigma\cap u_{\delta}\cap\mathrm{d}(c_{0},c_{1})\sub\mathrm{d}(a,b).$

For each $\delta\in\mathrm{d}(a,b)\cap\mathrm{d}(c_{0},c_{1})\cap\Sigma$
let $u_{\delta}$ be as in Claim I and define 
\begin{align*}
u:= & \bigcup\{u_{\delta}:\delta\in\mathrm{d}(a,b)\cap\mathrm{d}(c_{0},c_{1})\cap\Sigma\};\\
v:= & \mathrm{e}(a,b)\cup\mathrm{e}(c_{0},c_{1}).
\end{align*}
Note that $\mathcal{C}:=\{u,v\}$ is a $\te$-cover of $\Sigma$.
Indeed, if $\theta\in\Sigma\setminus v$, then $\theta\in\mathrm{d}(a,b)\cap\mathrm{d}(c_{0},c_{1})\cap\Sigma$,
and hence $\theta\in u_{\theta}\sub u$. 

We show next that $\mathcal{C}$ is compatible with $\S|_{\Sigma}$.
It is easy to see that $c_{0}$ is a solution for $\S|_{\Sigma\cap v}$.
Suppose $\theta\in\Sigma\cap u_{\delta}$ for some $\delta\in\mathrm{d}(a,b)\cap\mathrm{d}(c_{0},c_{1})\cap\Sigma$.
If $\theta\in\mathrm{e}(c_{0},c_{1})$, then $\langle c_{1},\S(\theta)\rangle\in\theta$.
Otherwise, $\theta\in\Sigma\cap u_{\delta}\cap\mathrm{d}(c_{0},c_{1})\sub\mathrm{d}(a,b)$,
and thus $\S(\theta)=c_{1}$. This shows that $c_{1}$ solves $\S|_{\Sigma\cap u}$.

Finally, since $\Sigma$ is a global spectrum, it is $\te$-complete
by Lemma~\ref{lem:global iff te-complete}. Hence, there is a solution
$s\in A$ for $\S|_{\Sigma}$. Clearly, $s$ also solves $\S$.
\end{proof}

\section{Global representation for quasivarieties with a near-unanimity term\label{sec:general-global-representation}}

We now apply the framework developed in the previous sections to obtain
a general global representation result for quasivarieties with a near-unanimity
term. The key idea for constructing a global spectrum $\Sigma$ for
an algebra $\mathbf{A}$ is the following. Start with a $\ted$-closed
set $\Gamma$ containing all completely meet-irreducible congruences
of $\mathbf{A}$, and enlarge it by adding a lower bound for each
$r$-element subset of $\Gamma$, where $r+1$ is the arity of the
near-unanimity term. The restriction to $\Gamma$ of any $\te$-system
on $\Sigma$ is then automatically $r$-solvable, and hence solvable
by the infinitary Baker--Pixley theorem.
\begin{thm}
\label{thm:espectro global para NU}Suppose $\mathcal{Q}$ has a $(r+1)$-ary
NU term and let $\mathcal{S}\subseteq\mathcal{Q}$ be an almost universal
class containing $\mathcal{Q}_{\mathrm{RSI}}$. Then, for every $\mathbf{A}\in\mathcal{Q}$
\[
\con_{\mathcal{S}}(\mathbf{A})\cup\{\bigcap\Omega:\Omega\sub\con_{\mathcal{S}}(\mathbf{A}),\left|\Omega\right|\in[2,r]\text{ and }\Omega\text{ is not a CRS}\}
\]
is a global spectrum of $\mathbf{A}$.
\end{thm}

\begin{proof}
Let us write $\Sigma$ for the set of congruences in the statement.
Fix a $\tau_{\mathrm{e}}$-system $\mathsf{S}$ on $\Sigma$. By (2)
of Lemma \ref{lem:CMI alcanza}, to prove that $\S$ is solvable,
it suffices to show that $\mathsf{S}|_{\con_{\mathcal{S}}(\mathbf{A})}$
is solvable. Moreover, item (2) of Lemma~\ref{lem:compacidad 1}
says that $\con_{\mathcal{S}}(\mathbf{A})$ is $\tau_{\mathrm{e}}$-compact.
So, in view of Theorem~\ref{thm:BP congruencial}, it is enough to
show that $\mathsf{S}|_{\con_{\mathcal{S}}(\mathbf{A})}$ is $r$-solvable.
Before we proceed to establish this fact, let us point out that $\S$
is a $\sqcup$-system, due to (1) of Lemma~\ref{lem:CMI alcanza}.

Let $\Lambda\sub\con_{\mathcal{S}}(\mathbf{A})$ with $\left|\Lambda\right|\leq r$.
If $\Lambda$ is a CRS, then $\mathsf{S}|_{\Lambda}$ is solvable.
Otherwise, we have that $\bigcap\Lambda\in\Sigma$, and hence 
\[
\langle\mathsf{S}(\bigcap\Lambda),\mathsf{S}(\lambda)\rangle\in\bigcap\Lambda\sqcup\lambda=\lambda\quad\text{for all }\lambda\in\Lambda.
\]
So $\mathsf{S}(\bigcap\Lambda)$ is a solution for $\mathsf{S}|_{\Lambda}$.
It follows that $\mathsf{S}|_{\con_{\mathcal{S}}(\mathbf{A})}$ is
$r$-solvable, as desired. Thus, $\Sigma$ is $\te$-complete. Moreover,
since $\mathcal{S}$ contains $\mathcal{Q}_{\mathrm{RSI}}$, we have
$\bigcap\Sigma=\Delta$. Therefore, Lemma~\ref{lem:global iff te-complete}
implies that $\Sigma$ is a global spectrum of $\mathbf{A}$.
\end{proof}

A class $\mathcal{F}\sub\mathcal{Q}$ is said to \emph{globally represent}
$\mathcal{Q}$ if every member of $\mathcal{Q}$ is isomorphic to
a global subdirect product with factors in $\mathbb{I}(\mathcal{F})$.
\begin{thm}
\label{thm:representacion global para NU}Let $\mathcal{Q}$ be a
quasivariety with a $(r+1)$-ary NU term, and let $\mathcal{K}:=\mathbb{P}_{u}(\mathcal{Q}_{\mathrm{RSI}})$.
Then the class
\[
\mathbb{S}(\mathcal{K})\cup\{\mathbf{F}\leq\mathbf{A}_{1}\times\cdots\times\mathbf{A}_{r}:\text{each }\mathbf{A}_{i}\in\mathcal{K}\text{ and }\{\ker\pi^{F}_{i}:i\in[1,r]\}\text{ is not a CRS}\}
\]
globally represents $\mathcal{Q}$.
\end{thm}

\begin{proof}
Let $\mathcal{F}$ denote the class displayed in the statement. Applying
Theorem~\ref{thm:espectro global para NU} with $\mathcal{S}=\mathbb{ISP}_{u}(\mathcal{Q}_{\mathrm{RSI}})$,
we obtain that $\mathcal{Q}$ is globally represented by the class
\[
\mathcal{F}':=\mathcal{S}\cup\{\mathbf{A}/\bigcap\Lambda:\mathbf{A}\in\mathcal{Q},\Lambda\sub\con_{\mathcal{S}}(\mathbf{A}),\left|\Lambda\right|\in[2,r]\text{ and }\Lambda\text{ is not a CRS}\},
\]
so it suffices to show that $\mathcal{F}'\sub\mathbb{I}(\mathcal{F})$.
Fix $\mathbf{F}\in\mathcal{F}'$. Since $\mathcal{S}=\mathbb{IS}(\mathcal{K})\sub\mathbb{I}(\mathcal{F})$
, we assume 
\[
\mathbf{F}=\mathbf{A}/(\gamma_{1}\cap\cdots\cap\gamma_{n}),
\]
for some $\mathbf{A}\in\mathcal{Q}$, $\gamma_{1},\dots,\gamma_{n}\in\con_{\mathcal{S}}(\mathbf{A})$,
$n\leq r$ and $\{\gamma_{1},\dots,\gamma_{n}\}$ not a CRS. If $n<r$
put $\gamma_{k}:=\gamma_{n}$ for $k\in[n+1,r]$, and note that $\mathbf{F}$
embeds in $\mathbf{A}/\gamma_{1}\times\cdots\times\mathbf{A}/\gamma_{r}$
via the natural map. Since $\mathbf{A}/\gamma_{1},\ldots,\mathbf{A}/\gamma_{r}\in\mathbb{IS}(\mathcal{K})$,
there are $\mathbf{A}_{1},\ldots,\mathbf{A}_{r}\in\mathcal{K}$ and
$\mathbf{G}\leq\mathbf{A}_{1}\times\cdots\times\mathbf{A}_{r}$ such
that $\mathbf{F}\cong\mathbf{G}$, and, under this isomorphism, $\delta_{i}:=\gamma_{i}/(\gamma_{1}\cap\cdots\cap\gamma_{r})$
corresponds to $\ker\pi^{G}_{i}$ for all $i$. As $\{\gamma_{1},\dots,\gamma_{r}\}$
is not a CRS, Lemma~\ref{lem:correspondencia} implies that neither
is $\{\delta_{1},\ldots,\delta_{r}\}$, and hence $\{\ker\pi^{G}_{i}:i\in[1,r]\}$
is not a CRS. Thus $\mathbf{G}\in\mathcal{F}$, which implies $\mathbf{F}\in\mathbb{I}(\mathcal{F})$.
\end{proof}

\begin{cor}
\label{cor:localizacion de RGI para nu a secas}Suppose that $\mathcal{Q}$
has a $(r+1)$-ary NU term. If $\mathbf{F}\in\mathcal{Q}_{\mathrm{RGI}}$,
then either 
\[
\mathbf{F}\in\mathbb{ISP}_{u}(\mathcal{Q}_{\mathrm{RSI}})
\]
or there exist $\mathbf{A}_{1},\ldots,\mathbf{A}_{r}\in\mathbb{P}_{u}(\mathcal{Q}_{\mathrm{RSI}})$
and 
\[
\mathbf{G}\leq\mathbf{A}_{1}\times\cdots\times\mathbf{A}_{r}
\]
such that $\mathbf{F}\cong\mathbf{G}$ and $\{\ker\pi^{G}_{i}:i\in[1,r]\}$
is not a CRS.
\end{cor}

\begin{proof}
Immediate from Theorem~\ref{thm:representacion global para NU}.
\end{proof}

\begin{rem}
In~\cite[Problem~17]{vaggione_2022}, the following general problem
is posed: find a class $\mathcal{F}\sub\mathcal{Q}$, as small and
manageable as possible, that globally represents $\mathcal{Q}$. Observe
that the optimal situation would be to have $\mathbb{I}(\mathcal{F})=\mathcal{Q}_{\mathrm{RGI}}$.
Theorem~19 of \cite{vaggione_2022} asserts that if $\mathcal{Q}$
has a $(r+1)$-ary near-unanimity term, then the class 
\[
\mathcal{F}=\mathbb{S}(\{\mathbf{A}_{1}\times\cdots\times\mathbf{A}_{r}:\mathbf{A}_{i}\in\mathbb{P}_{u}(\mathcal{Q}_{\mathrm{RSI}}),\ i\in[1,r]\})
\]
globally represents $\mathcal{Q}$. If $\mathbf{F}\in\mathcal{F}\setminus\mathbb{ISP}_{u}(\mathcal{Q}_{\mathrm{RSI}})$,
say $\mathbf{F}\leq\mathbf{A}_{1}\times\cdots\times\mathbf{A}_{r}$,
with each $\mathbf{A}_{i}\in\mathbb{P}_{u}(\mathcal{Q}_{\mathrm{RSI}})$,
and $\{\ker\pi^{F}_{i}:1\leq i\leq r\}$ is a CRS, then Remark~\ref{rem:CRS implica descomponible}
says that $\mathbf{F}$ is globally decomposable. Thus, Theorem~\ref{thm:representacion global para NU}
yields a refinement of Vaggione's result, since it shows that all
these algebras can be discarded from $\mathcal{F}$ while still obtaining
a class that globally represents the quasivariety.
\end{rem}

We conclude this section by specializing Theorem~\ref{thm:representacion global para NU}
to arithmetical varieties. The resulting statement, already obtained
by Vaggione in \cite[Corollary~2 to Theorem~6.2]{Vaggione1992}, takes
a particularly sharp form; here it follows immediately from our general
representation theorem.
\begin{cor}
If $\mathcal{V}$ is an arithmetical variety, then $\mathbb{SP}_{u}(\mathcal{V}_{\mathrm{SI}})$
globally represents $\mathcal{V}$.
\end{cor}

\begin{proof}
Since every arithmetical variety has a majority term \cite[Theorem~II.12.5]{BurrisSankappanavar2012},
the result follows immediately from Theorem~\ref{thm:representacion global para NU},
applied with $r=2$.
\end{proof}

The corollary applies to a wide range of familiar algebraic structures.
In particular, the variety of residuated lattices is arithmetical
\cite{JipsenTsinakis2002}, and hence the result applies to all its
subvarieties and standard expansions, including the varieties of $\ell$-groups,
Heyting algebras, BL-algebras, MTL-algebras, and MV-algebras. Further
classical examples of arithmetical varieties include $f$-rings and
vector groups. It is also worth noting that many of the mentioned
varieties have an almost universal class of FSIs. Hence, in this case,
the FSIs already globally represent the variety.

The generality of Theorem~\ref{thm:representacion global para NU}
comes at a price: the resulting representing class may be unwieldy
and may contain globally decomposable algebras. Under additional assumptions
on the ambient quasivariety, however, this class can be substantially
refined. This is precisely what we do in the sequel.

\section{RCD globally indecomposable algebras of finite subdirect width\label{sec:finite-subdirect-width}}

Let us recall some lattice-theoretic definitions. Let $\mathbf{L}$
be a lattice. For $a,b\in L$ we say that $b$ is a \emph{cover} of
$a$ if $a<b$ and there is no element strictly between $a$ and $b$.
We write $\cov(a)$ to denote the set of all covers of $a$. If $\mathbf{L}$
has bottom element $0$, an \emph{atom} of $\mathbf{L}$\emph{ }is
an element in $\cov(0)$. The lattice $\mathbf{L}$ is \emph{atomic}
if for every $b\in L\setminus\{0\}$ there is an atom $a$ of $\mathbf{L}$
with $a\leq b$. For a positive integer $n$, we say that a lattice
is $n$\emph{-atomic} if it is atomic and has exactly $n$ atoms.

Note that $\mathbf{A}\in\mathcal{Q}_{\mathrm{RSI}}$ if and only if
$\con_{\mathcal{Q}}(\mathbf{A})$ is 1-atomic. Given $\gamma\in\cmic_{\mathcal{Q}}(\mathbf{A})$
we write $\hat{\gamma}$ to denote the unique cover of $\gamma$ in
$\con_{\mathcal{Q}}(\mathbf{A})$.
\begin{lem}
\label{lem:atomos de Con A inducido por un conjunto finito de CMIs}Let
$\mathbf{A}\in\mathcal{Q}$ be RCD and suppose $\Omega\sub\cmic_{\mathcal{Q}}(\mathbf{A})$
is finite, its members are pairwise incomparable, and $\bigcap\Omega=\Delta$.
Then, the lattice $\con_{\mathcal{Q}}(\mathbf{A})$ is $\left|\Omega\right|$-atomic,
and $\hat{\gamma}\cap\bigcap(\Omega\setminus\{\gamma\})$ is an atom
of $\con_{\mathcal{Q}}(\mathbf{A})$ for each $\gamma\in\Omega$.
\end{lem}

\begin{proof}
In the following argument, we rely (without explicit mention) on the
fact that every congruence in $\cmic_{\mathcal{Q}}(\mathbf{A})$ is
meet-prime.

For each $\gamma\in\Omega$ put
\[
\mu_{\gamma}:=\hat{\gamma}\cap\bigcap(\Omega\setminus\{\gamma\}).
\]
Fix $\delta\in\Omega$. Since $\hat{\delta}\nsubseteq\delta$ and
the members of $\Omega$ are pairwise incomparable, it follows that
$\mu_{\delta}\nsubseteq\delta$. In particular, we have $\mu_{\delta}\neq\Delta.$
Let 
\[
\Gamma_{\delta}:=\cmic_{\mathcal{Q}}(\mathbf{A})\setminus\{\delta\};
\]
we claim that $\mu_{\delta}=\bigcap\Gamma_{\delta}$. Indeed, suppose
$\gamma\in\cmic_{\mathcal{Q}}(\mathbf{A})\setminus\{\delta\}$. Then,
as $\bigcap\Omega\sub\gamma$, we have that either $\delta\subsetneq\gamma$
or $\bigcap(\Omega\setminus\{\delta\})\sub\gamma$; hence $\mu_{\delta}\sub\gamma$.
The reverse inclusion follows from $\Omega\setminus\{\delta\}\sub\Gamma_{\delta}$
and $\hat{\delta}=\bigcap\{\gamma\in\cmic_{\mathcal{Q}}(\mathbf{A}):\hat{\delta}\sub\gamma\}$.
Indeed, since $\hat{\delta}\nsubseteq\delta$, every congruence in
the latter family belongs to $\Gamma_{\delta}$, and hence $\bigcap\Gamma_{\delta}\sub\hat{\delta}\cap\bigcap(\Omega\setminus\{\delta\})=\mu_{\delta}$.

Now, if $\theta$ is a non-diagonal congruence of $\mathbf{A}$, then
the set 
\[
\Gamma_{\theta}:=\{\gamma\in\cmic_{\mathcal{Q}}(\mathbf{A}):\theta\sub\gamma\}
\]
must omit some $\gamma$ in $\Omega$, and hence $\Gamma_{\theta}\sub\Gamma_{\gamma}$.
That is, $\mu_{\gamma}\sub\theta$ for some $\gamma\in\Omega$. It
follows that $\con_{\mathcal{Q}}(\mathbf{A})$ is atomic and its atoms
are the $\mu_{\gamma}$'s. Moreover, these must be pairwise distinct,
since for $\gamma\neq\delta$ we have $\mu_{\gamma}\sub\delta$ and
$\mu_{\delta}\nsubseteq\delta$.
\end{proof}

The preceding lemma relates finite irredundant families of relatively
completely meet-irreducible congruences to the atomic structure of
the relative congruence lattice. This motivates the following invariant.

For $\mathbf{A}\in\mathcal{Q}$, define the \emph{relative subdirect
width} of $\mathbf{A}$, denoted by $\sdw_{\mathcal{Q}}(\mathbf{A})$,
as the least cardinal $\kappa$ such that $\mathbf{A}$ has a subdirect
representation with exactly $\kappa$ factors, each of which is relatively
subdirectly irreducible. Equivalently, 
\[
\sdw_{\mathcal{Q}}(\mathbf{A})=\min\{|\Gamma|:\Gamma\subseteq\cmic_{\mathcal{Q}}(\mathbf{A})\text{ and }\bigcap\Gamma=\Delta\}.
\]

A finite subdirect product 
\[
\mathbf{A}\leq_{\mathrm{sd}}\mathbf{A}_{1}\times\cdots\times\mathbf{A}_{n}
\]
is irredundant if, for every $i\in[1,n]$, the induced map $\mathbf{A}\rightarrow\prod_{j\neq i}\mathbf{A}_{j}$
is not injective. Equivalently, no projection onto a proper subproduct
is injective. In this case, we write 
\[
\mathbf{A}\leq_{\mathrm{isd}}\mathbf{A}_{1}\times\cdots\times\mathbf{A}_{n}.
\]

\begin{lem}
\label{lem:equivalencias RCD ancho finito}Let $\mathbf{F}\in\mathcal{Q}$
be RCD, and let $n>0$. The following are equivalent:
\begin{enumerate}
\item $\sdw_{\mathcal{Q}}(\mathbf{F})=n$.
\item $\mathbf{F}$ has an irredundant subdirect representation with exactly
$n$ relatively subdirectly irreducible factors.
\item $\con_{\mathcal{Q}}(\mathbf{F})$ is $n$-atomic.
\end{enumerate}
\end{lem}

\begin{proof}
(1)$\Rightarrow$(2). This follows directly from the definitions.

(2)$\Rightarrow$(3). Note that if (2) holds, then there is an $n$-element
set $\Omega\sub\cmic_{\mathcal{Q}}(\mathbf{F})$ such that $\bigcap\Omega=\Delta$,
and the congruences in $\Omega$ are pairwise incomparable. So (3)
now follows from Lemma~\ref{lem:atomos de Con A inducido por un conjunto finito de CMIs}.

(3)$\Rightarrow$(1). Let $\mu_{1},\ldots,\mu_{n}$ be the atoms of
$\con_{\mathcal{Q}}(\mathbf{F})$. For each $i$ choose $\gamma_{i}\in\cmic_{\mathcal{Q}}(\mathbf{F})$
such that $\mu_{i}\nsubseteq\gamma_{i}$. Then $\bigcap^{n}_{i=1}\gamma_{i}=\Delta$,
and therefore $\sdw_{\mathcal{Q}}(\mathbf{F})\leq n$. If $\sdw_{\mathcal{Q}}(\mathbf{F})=m<n$,
the already established implication (1)$\Rightarrow$(3) would say
that $\con_{\mathcal{Q}}(\mathbf{F})$ is $m$-atomic, which is not
the case. Hence $\sdw_{\mathcal{Q}}(\mathbf{F})=n$.
\end{proof}

\begin{lem}
\label{lem: minoriza esp glo rcd implica crs}Let $\mathbf{A}$ be
RCD and let $\Sigma\sub\con_{\mathcal{Q}}(\mathbf{A})$ be $\te$-complete,
with $\bigcap\Sigma=\Delta$. Then, every finite $\Omega\sub\con_{\mathcal{Q}}(\mathbf{A})$
that minorizes $\Sigma$ is a CRS.
\end{lem}

\begin{proof}
Let $\Sigma$ be as in the statement and suppose $\Omega\sub\con_{\mathcal{Q}}(\mathbf{A})$
is finite and minorizes $\Sigma$. Fix a $\sqcup$-system $\T$ on
$\Omega$. For each $\mu\in\Omega$ choose a compact congruence $\mu^{*}\sub\mu$
such that 
\[
\langle\T(\upsilon_{1}),\T(\upsilon_{2})\rangle\in\upsilon^{*}_{1}\sqcup\upsilon^{*}_{2}\quad\text{for all }\upsilon_{1},\upsilon_{2}\in\Omega.
\]
Note that, since $\Omega$ is finite, the congruences $\mu^{*}$ may
be chosen pairwise distinct. Let $\Omega^{*}:=\{\mu^{*}:\mu\in\Omega\}$,
and observe that $\Omega^{*}\mins\Omega\mins\Sigma$. Define $\S:\Omega^{*}\rightarrow A$
by $\S(\mu^{*}):=\T(\mu)$. Then, $\S$ is a $\sqcup$-system and
$\S\mins\T.$

For each $\mu^{*}\in\Omega^{*}$ take a finite $F_{\mu}\sub\mu^{*}$
such that $\mu^{*}=\cg^{\mathbf{A}}_{\mathcal{Q}}(F_{\mu})$, and
define
\[
u_{\mu}:=\bigcap_{p\in F_{\mu}}\mathrm{e}(p).
\]
Note that, as $\Omega^{*}\mins\Sigma$, we have
\[
\Sigma\sub\bigcup_{\mu\in\Omega}u_{\mu},
\]
and hence
\[
\overline{\Sigma}^{\mathrm{ed}}\sub\overline{\bigcup_{\mu\in\Omega}u_{\mu}}^{\mathrm{ed}}=\bigcup_{\mu\in\Omega}\overline{u_{\mu}}^{\mathrm{ed}}=\bigcup_{\mu\in\Omega}u_{\mu}.
\]
That is, $\mathcal{C}:=\{u_{\mu}:\mu\in\Omega\}$ covers $\overline{\Sigma}^{\mathrm{ed}}$,
and it follows that $\Omega^{*}\mins\overline{\Sigma}^{\mathrm{ed}}$.
So, by Lemma \ref{lem:extender sistema para arriba}, there is a $\sqcup$-system
$\S^{+}$ on $\overline{\Sigma}^{\mathrm{ed}}$ satisfying $\S\mins\S^{+}$.
We claim that $\S(\mu^{*})$ is a solution for $\S^{+}|_{u_{\mu}}$
for every $\mu\in\Omega$. Indeed, fix $\mu\in\Omega$ and let $\theta\in\overline{\Sigma}^{\mathrm{ed}}\cap u_{\mu}$.
Then $\mu^{*}\sub\theta$, and hence $\langle\S(\mu^{*}),\S^{+}(\theta)\rangle\in\theta$.
So $\mathcal{C}$ is a $\te$-cover of $\overline{\Sigma}^{\mathrm{ed}}$
compatible with $\S^{+}$, which entails that $\S^{+}$ is a $\te$-system.
Thus, since $\overline{\Sigma}^{\mathrm{ed}}$ is $\te$-complete
by Lemma \ref{lem:global en la clausura}, there is a solution $a\in A$
for $\S^{+}$. We show that $a$ is a solution for $\S$. Take $\mu^{*}\in\Omega^{*}$,
and let $\gamma\in\cmic_{\mathcal{Q}}(\mathbf{A})$ such that $\mu^{*}\sub\gamma$.
Since $\bigcap\overline{\Sigma}^{\mathrm{ed}}=\Delta\sub\gamma$ and
$\gamma$ is meet-prime, by Lemma \ref{lem:completamente meet-prime},
there is $\theta\in\overline{\Sigma}^{\mathrm{ed}}$ such that $\theta\sub\gamma$.
Furthermore, since $\Omega^{*}\mins\overline{\Sigma}^{\mathrm{ed}},$
there is $\upsilon^{*}\in\Omega^{*}$ with $\upsilon^{*}\sub\theta$.
Now, as $\S\mins\Sup$ and $a$ is a solution for $\S^{+}$, we have
\[
\langle\S(\upsilon^{*}),\S^{+}(\theta)\rangle\in\theta\text{ and }\langle a,\S^{+}(\theta)\rangle\in\theta;
\]
hence 
\[
\langle a,\S(\upsilon^{*})\rangle\in\theta\sub\gamma.
\]
Also, since $\S$ is an $\sqcup$-system, 
\[
\langle\S(\mu^{*}),\S(\upsilon^{*})\rangle\in\mu^{*}\sqcup\upsilon^{*}\sub\gamma,
\]
and it follows that
\[
\langle a,\S(\mu^{*})\rangle\in\gamma.
\]
Thus, since $\gamma$ is an arbitrary congruence in $\cmic_{\mathcal{Q}}(\mathbf{A})$
containing $\mu^{*}$, we conclude that $\langle a,\S(\mu^{*})\rangle\in\mu^{*}$.
So, $a$ is a solution for $\S$. Finally, since $\S\mins\T$, $a$
is a solution for $\T$ by Lemma \ref{lem:minos}.
\end{proof}

We can now characterize RGI algebras whose relative congruence lattice
is distributive and has finitely many atoms. The following proposition
generalizes \cite[Corollary 7]{Vaggione2019}.
\begin{prop}
\label{prop:RGI RCD ancho finite}Let $\mathbf{F}$ be an RCD algebra
in $\mathcal{Q}$ such that $\con_{\mathcal{Q}}(\mathbf{F})$ is $n$-atomic
for $n\geq2$. The following are equivalent:
\begin{enumerate}
\item The algebra $\mathbf{F}$ is relatively globally indecomposable.
\item The set of atoms of $\con_{\mathcal{Q}}(\mathbf{F})$ is not a CRS.
\end{enumerate}
\end{prop}

\begin{proof}
Let $\Omega$ be the set of atoms of $\con_{\mathcal{Q}}(\mathbf{F})$.
We prove that $\Omega$ is a CRS if and only if $\mathbf{F}$ has
a nontrivial global representation with factors in $\mathcal{Q}$.

Suppose first that $\Omega$ is a CRS, and put 
\[
\Sigma:=\{\theta\sqcup\delta:\theta,\delta\in\Omega\}.
\]
By Lemma~\ref{lem:CRS finito sii los joins son global spec}, the
set $\Sigma$ is $\te$-complete. Since $\Omega\sub\Sigma$, we have
$\bigcap\Sigma=\Delta.$ Moreover, $\Delta\notin\Sigma$, since every
member of $\Sigma$ contains an atom. Hence, Lemma~\ref{lem:global iff te-complete}
implies that $\Sigma$ is a global spectrum, and therefore yields
a nontrivial global representation of $\mathbf{F}$.

Conversely, suppose that $\mathbf{F}$ has a nontrivial global representation
with factors in $\mathcal{Q}$, and let $\Sigma\sub\con_{\mathcal{Q}}(\mathbf{F})$
be its induced spectrum. Then $\Sigma$ is $\te$-complete, $\bigcap\Sigma=\Delta,$
and $\Delta\notin\Sigma$. Since $\con_{\mathcal{Q}}(\mathbf{F})$
is atomic, $\Omega$ minorizes $\Sigma$. Therefore, Lemma~\ref{lem: minoriza esp glo rcd implica crs}
implies that $\Omega$ is a CRS.
\end{proof}

\subsection*{g-algebras}

We next translate the preceding congruence-theoretic characterization
of RGIs into a combinatorial one for finite irredundant subdirect
representations over relatively simple factors.

Let $n\geq2$ , let $S_{1},\ldots,S_{n}$ be nonempty sets, and let
$G\subseteq S_{1}\times\cdots\times S_{n}$. A tuple 
\[
\langle s_{1},\ldots,s_{n}\rangle\in S_{1}\times\cdots\times S_{n}
\]
 is \emph{almost in} $G$ if $\langle s_{1},\ldots,s_{n}\rangle\notin G$
and, for every $i\in[1,n]$, there exists $a_{i}\in S_{i}$ such that
\[
\langle s_{1},\ldots,s_{i-1},a_{i},s_{i+1},\ldots,s_{n}\rangle\in G.
\]
Note that if $G\sub S_{1}\times S_{2}$, there is a tuple almost in
$G$ if and only if G is a proper subdirect relation.
\begin{lem}
\label{lem:RGI semisimple RCD}Suppose $\mathbf{G}$ is RCD, and 
\[
\mathbf{G}\leq_{\mathrm{isd}}\mathbf{S}_{1}\times\ldots\times\mathbf{S}_{n}
\]
 for some $n\geq2$ and $\mathbf{S}_{1},\ldots,\mathbf{S}_{n}\in\mathcal{Q}_{\mathrm{RS}}$.
The following are equivalent:
\begin{enumerate}
\item $\mathbf{G}$ is RGI.
\item There is a tuple that is almost in $G$.
\end{enumerate}
\end{lem}

\begin{proof}
For each $i\in[1,n]$ put $\gamma_{i}:=\ker\pi^{\mathbf{G}}_{i}$.
Note that $\Omega:=\{\gamma_{i}:i\in[1,n]\}$ satisfies all assumptions
of Lemma~\ref{lem:atomos de Con A inducido por un conjunto finito de CMIs},
so the lattice $\con_{\mathcal{Q}}(\mathbf{G})$ is $n$-atomic, and
$\mu_{i}:=\hat{\gamma}_{i}\cap\bigcap_{j\neq i}\gamma_{j}$ is an
atom for every $i\in[1,n]$. Since each $\gamma_{i}$ is a maximal
$\mathcal{Q}$-congruence of $\mathbf{G}$, we have $\mu_{i}=\bigcap_{j\neq i}\gamma_{j}$.
Moreover, distributivity yields $\mu_{i}\sqcup\mu_{j}=\bigcap_{k\neq i,j}\gamma_{k}$
whenever $i\neq j$.

We claim that there is an unsolvable $\sqcup$-system on $\{\mu_{i}:i\in[1,n]\}$
if and only if there is a tuple in $S_{1}\times\cdots\times S_{n}$
that is almost in $G$.

Let $\S\colon\mu_{1},\ldots,\mu_{n}\mapsto\mathbf{g}_{1},\ldots,\mathbf{g}_{n}$
be a $\sqcup$-system. For each $k\in[1,n]$, define $\mathbf{s}(k):=\mathbf{g}_{i}(k)$
for any $i\neq k$. This is well defined: if $i,j\neq k$, then $\langle\mathbf{g}_{i},\mathbf{g}_{j}\rangle\in\mu_{i}\sqcup\mu_{j}\subseteq\gamma_{k}$,
and hence $\mathbf{g}_{i}(k)=\mathbf{g}_{j}(k)$.

For each $i\in[1,n]$, the tuples $\mathbf{s}$ and $\mathbf{g}_{i}$
agree at every coordinate other than $i$. We now observe that $\S$
is solvable if and only if $\mathbf{s}\in G$. Indeed, if $\mathbf{s}\in G$,
then $\langle\mathbf{s},\mathbf{g}_{i}\rangle\in\mu_{i}$ for every
$i\in[1,n]$, so $\mathbf{s}$ is a solution of $\S$. Conversely,
if $\mathbf{g}\in G$ is a solution of $\S$, then, for every $k\in[1,n]$,
choosing $i\neq k$ gives $\mathbf{g}(k)=\mathbf{g}_{i}(k)=\mathbf{s}(k)$.
Thus $\mathbf{g}=\mathbf{s}$, and consequently $\mathbf{s}\in G$.
It follows that if $\S$ is unsolvable, then $\mathbf{s}\notin G$,
and therefore $\mathbf{s}$ is almost in $G$.

Conversely, suppose that $\mathbf{s}\in S_{1}\times\cdots\times S_{n}$
is almost in $G$. For each $i\in[1,n]$, choose $\mathbf{g}_{i}\in G$
that differs from $\mathbf{s}$ only at the $i$th coordinate. Then,
for $i\neq j$, the tuples $\mathbf{g}_{i}$ and $\mathbf{g}_{j}$
agree outside the coordinates $i$ and $j$, and hence $\langle\mathbf{g}_{i},\mathbf{g}_{j}\rangle\in\bigcap_{k\neq i,j}\gamma_{k}=\mu_{i}\sqcup\mu_{j}$.
Thus $\mu_{1},\ldots,\mu_{n}\mapsto\mathbf{g}_{1},\ldots,\mathbf{g}_{n}$
is a $\sqcup$-system. Its associated tuple is $\mathbf{s}$, so,
since $\mathbf{s}\notin G$, the preceding argument shows that this
system is unsolvable.

The result now follows from Proposition~\ref{prop:RGI RCD ancho finite}.
\end{proof}

Lemma~\ref{lem:RGI semisimple RCD} motivates the following definition.
An algebra $\mathbf{F}\in\mathcal{Q}$ is a relative $g$-algebra\footnote{The terminology is motivated by the $g$-relations of \cite{Vaggione2019}. }
if it is RCD and admits an irredundant subdirect representation 
\[
\mathbf{F}\cong\mathbf{G}\leq_{\mathrm{isd}}\mathbf{S}_{1}\times\cdots\times\mathbf{S}_{n}
\]
for $n\geq2$ and $\mathbf{S}_{1},\ldots,\mathbf{S}_{n}\in\mathcal{Q}_{\mathrm{RS}}$,
such that some tuple is almost in $G$.
\begin{rem}
\label{rem:g-algebras de ancho 2}Note that the $g$-algebras of width
two are exactly those admitting an irredundant subdirect representation
with two simple factors whose image is properly contained in their
product. Indeed, for a subdirect relation $G\sub S_{1}\times S_{2}$,
the existence of a tuple almost in $G$ is equivalent to $G\neq S_{1}\times S_{2}$.
\end{rem}

\begin{cor}
\label{cor:toda g-algebra es RGI}Every relative $g$-algebra is RGI.
\end{cor}

\begin{proof}
Immediate from Lemma~\ref{lem:RGI semisimple RCD}. 
\end{proof}

\section{Quasivarieties whose RSI form an almost universal class\label{sec:almost-universal-rsi}}

The aim of this section is to prove that every RCD algebra in a quasivariety
with a NU term and an almost universal class of RSIs is RGI-spectral.
We are already in a position to characterize the globally indecomposables
involved.
\begin{lem}
\label{lem:carac RGI RCD para NU y clase universal de SI}Suppose
that $\mathcal{Q}$ has a $(r+1)$-ary NU term and that $\mathcal{Q}_{\mathrm{RSI}}$
is almost universal. The following hold:
\begin{enumerate}
\item For an RCD algebra $\mathbf{F}\in\mathcal{Q}$ the following are equivalent:
\begin{enumerate}
\item $\mathbf{F}$ is RGI.
\item $\con_{\mathcal{Q}}(\mathbf{F})$ is $n$-atomic for some $n\in[1,r]$,
and if $n\geq2$ its atoms do not form a CRS.
\end{enumerate}
\item If $\mathbf{A}\in\mathcal{Q}$ is RCD, then for $\theta\in\con_{\mathcal{Q}}(\mathbf{A})$
the following are equivalent:
\begin{enumerate}
\item $\theta$ is an RGI-congruence of $\mathbf{A}$.
\item $\theta$ is CMI or $\left|\cov(\theta)\right|\in[2,r]$ and $\cov(\theta)$
is not a CRS. 
\end{enumerate}
\end{enumerate}
\end{lem}

\begin{proof}
(1). Suppose $\mathbf{F}\in\mathcal{Q}_{\mathrm{RGI}}$ is RCD.

(a)$\Rightarrow$(b). If $\mathbf{F}\in\mathcal{Q}_{\mathrm{RSI}}$,
we are done; so assume $\mathbf{F}\notin\mathcal{Q}_{\mathrm{RSI}}$.
Put $n:=\sdw_{\mathcal{Q}}(\mathbf{F})$; we claim that
\begin{equation}
n\in[2,r].\label{eq: sdw(F) acotado}
\end{equation}
Indeed, Corollary~\ref{cor:localizacion de RGI para nu a secas}
together with the fact that $\mathcal{Q}_{\mathrm{RSI}}$ is almost
universal imply 
\[
\mathbf{F}\cong\mathbf{G}\leq\mathbf{A}_{1}\times\ldots\times\mathbf{A}_{r}
\]
for some $\mathbf{A}_{1},\ldots,\mathbf{A}_{r}\in\mathcal{Q}^{+}_{\mathrm{RSI}}$
. As $\mathcal{Q}^{+}_{\mathrm{RSI}}$ is closed under subalgebras,
we have that 
\[
\mathbf{B}_{i}:=\pi_{i}[\mathbf{G}]\in\mathcal{Q}^{+}_{\mathrm{RSI}}\quad\text{for }i\in[1,r],
\]
hence 
\[
\mathbf{F}\hookrightarrow_{\mathrm{sd}}\mathbf{B}_{1}\times\ldots\times\mathbf{B}_{r},
\]
and it follows that $n\leq r$. Also, since $\mathbf{F}\in\mathcal{Q}_{\mathrm{RGI}}\setminus\mathcal{Q}_{\mathrm{RSI}}$,
we have $n\geq2$. Hence, (\ref{eq: sdw(F) acotado}) holds.

Lemma~\ref{lem:equivalencias RCD ancho finito} says that $\con_{\mathcal{Q}}(\mathbf{F})$
is $n$-atomic, and, by Proposition~\ref{prop:RGI RCD ancho finite},
the atoms of $\con_{\mathcal{Q}}(\mathbf{F})$ do not form a CRS.

(b)$\Rightarrow$(a). If $n=1$, then $\mathbf{F}\in\mathcal{Q}_{\mathrm{RSI}}\sub\mathcal{Q}_{\mathrm{RGI}}$.
For $n\geq2$ apply Proposition~\ref{prop:RGI RCD ancho finite}.

(2). This follows from (1) and Lemma~\ref{lem:correspondencia}.
\end{proof}

\begin{rem}
Recall that the congruence lattice of every algebra with a NU term
is distributive. Hence, when $\mathcal{Q}$ is a variety, the RCD
assumption in Lemma~\ref{lem:carac RGI RCD para NU y clase universal de SI}
is superfluous. The same observation applies to several results below. 
\end{rem}

The next two lemmas set up the proof of the main theorem in this section.
\begin{lem}
\label{lem:S r-soluble sobre CMI alcanza}Suppose that $\mathcal{Q}$
has a $(r+1)$-ary NU term and that $\mathcal{Q}_{\mathrm{RSI}}$
is almost universal. Let $\mathbf{A}\in\mathcal{Q}$, and let $\Sigma\sub\con_{\mathcal{Q}}(\mathbf{A})$
be $\te$-compact and CMI-saturated. If $\S$ is a $\te$-system on
$\Sigma$ such that $\S|_{[\Sigma)\cap\cmic_{\mathcal{Q}}(\mathbf{A})}$
is $r$-solvable, then $\S$ is solvable.
\end{lem}

\begin{proof}
Suppose $\Sigma$ is as in the statement and put $\Gamma:=\cmic_{\mathcal{Q}}(\mathbf{A})$.
Note that $[\Sigma)\cap\Gamma$ is $\te$-compact by Lemma \ref{lem:compacidad 1}.

Let $\S$ be a $\te$-system on $\Sigma$ such that $\S|_{[\Sigma)\cap\Gamma}$
is $r$-solvable. Then, as $[\Sigma)\cap\Gamma$ is $\te$-compact,
Theorem~\ref{thm:BP congruencial} implies that $\S|_{[\Sigma)\cap\Gamma}$
is solvable. So, by Lemma~\ref{lem:CMI alcanza}, $\S$ is solvable.
\end{proof}

\begin{lem}
\label{lem:almost solvable}Let $\mathbf{A}$ be RCD. Suppose $\Omega\sub\cmic_{\mathcal{Q}}(\mathbf{A})$
is finite and $\cov(\bigcap\Omega)$ is a CRS. Let $\S$ be a $\sqcup$-system
on $\Omega$, and for each $\delta\in\Omega$ let $\S_{\delta}$ be
the system on $(\Omega\setminus\{\delta\})\cup\{\hat{\delta}\}$ defined
by 
\begin{align*}
\S_{\delta}(\hat{\delta}):= & \S(\delta);\\
\S_{\delta}(\gamma):= & \S(\gamma)\text{ for }\gamma\in(\Omega\setminus\{\delta\}).
\end{align*}
If $\S_{\delta}$ is solvable for every $\delta\in\Omega$, then $\S$
is solvable.
\end{lem}

\begin{proof}
Replacing $\Omega$ by its subset of minimal elements if necessary,
we can assume without loss of generality that the members of $\Omega$
are pairwise incomparable. Moreover, by Lemma~\ref{lem:correspondencia}
it is enough to consider the case $\bigcap\Omega=\Delta$. Also, if
$\Omega$ is a singleton, the lemma is trivial, so we assume $\left|\Omega\right|\geq2$.
After these reductions, applying Lemma~\ref{lem:atomos de Con A inducido por un conjunto finito de CMIs}
to $\Omega$ yields that
\[
\mu_{\gamma}:=\hat{\gamma}\cap\bigcap(\Omega\setminus\{\gamma\})
\]
is an atom of $\con_{\mathcal{Q}}(\mathbf{A})$ for each $\gamma\in\Omega$,
and every atom of $\con_{\mathcal{Q}}(\mathbf{A})$ is of this form.
Then, $\{\mu_{\gamma}:\gamma\in\Omega\}=\cov(\bigcap\Omega)$ is a
CRS by hypothesis.

Fix a $\sqcup$-system $\S$ on $\Omega$ and for each $\delta\in\Omega$
let $\S_{\delta}$ be as in the statement. Assume that each $\S_{\delta}$
has a solution $a_{\delta}\in A$. That is, 
\begin{align}
\langle a_{\delta},\S(\delta)\rangle & \in\hat{\delta}\quad\text{for all }\delta\text{, and}\label{eq:soluciones}\\
\langle a_{\delta},\S(\gamma)\rangle & \in\gamma\quad\text{ for all }\gamma\neq\delta.\nonumber 
\end{align}
We claim that 
\begin{equation}
\langle a_{\gamma},a_{\delta}\rangle\in\mu_{\gamma}\sqcup\mu_{\delta}\quad\text{for all }\gamma,\delta.\label{eq:supremos}
\end{equation}
Fix $\gamma\neq\delta$; by distributivity, 

\begin{align*}
\mu_{\gamma}\sqcup\mu_{\delta} & =(\hat{\gamma}\cap\delta\cap\bigcap(\Omega\setminus\{\gamma,\delta\}))\sqcup(\hat{\delta}\cap\gamma\cap\bigcap(\Omega\setminus\{\gamma,\delta\}))\\
 & =((\hat{\gamma}\cap\delta)\sqcup(\hat{\delta}\cap\gamma))\cap\bigcap(\Omega\setminus\{\gamma,\delta\})\\
 & =(\hat{\gamma}\cap\hat{\delta})\cap(\gamma\sqcup\delta)\cap\bigcap(\Omega\setminus\{\gamma,\delta\}).
\end{align*}
We show that $\langle a_{\gamma},a_{\delta}\rangle$ belongs to $(\hat{\gamma}\cap\hat{\delta})$,
$(\gamma\sqcup\delta)$, and $\bigcap(\Omega\setminus\{\gamma,\delta\})$.
By (\ref{eq:soluciones}) we have $a_{\gamma}\stackrel{\hat{\gamma}}{\sim}\S(\gamma)\stackrel{\gamma}{\sim}a_{\delta}$
and $a_{\delta}\stackrel{\hat{\delta}}{\sim}\S(\delta)\stackrel{\delta}{\sim}a_{\gamma}$,
which entails
\[
\langle a_{\gamma},a_{\delta}\rangle\in\hat{\gamma}\cap\hat{\delta}.
\]
From (\ref{eq:soluciones}) and the fact that $\mathsf{S}$ is a $\sqcup$-system,
it follows that $a_{\gamma}\stackrel{\delta}{\sim}\S(\delta)\stackrel{\gamma\sqcup\delta}{\sim}\S(\gamma)\stackrel{\gamma}{\sim}a_{\delta}$.
Hence,
\[
\langle a_{\gamma},a_{\delta}\rangle\in\gamma\sqcup\delta.
\]
By (\ref{eq:soluciones}) we have that, if $\rho\in\Omega\setminus\{\gamma,\delta\}$,
then $\langle a_{\gamma},\S(\rho)\rangle,\langle a_{\delta},\S(\rho)\rangle\in\rho$.
Therefore,
\[
\langle a_{\gamma},a_{\delta}\rangle\in\bigcap(\Omega\setminus\{\gamma,\delta\}).
\]
So (\ref{eq:supremos}) holds, which says that the system $\T$ on
$\{\mu_{\gamma}:\gamma\in\Omega\}$ given by $\T(\mu_{\gamma}):=a_{\gamma}$
for all $\gamma$, is a $\sqcup$-system. Now, since $\{\mu_{\gamma}:\gamma\in\Omega\}$
is a CRS, there is a solution $a\in A$ for $\mathsf{T}$. We conclude
by showing that $a$ is a solution for $\mathsf{S}$. Fix $\gamma\in\Omega$,
and let $\delta\in\Omega\setminus\{\gamma\}$. Then, by (\ref{eq:soluciones}),
we have $\langle a_{\delta},\S(\gamma)\rangle\in\gamma$, and since
$a$ is a solution for $\T$, we have $\langle a,a_{\delta}\rangle\in\mu_{\delta}\sub\gamma$.
Hence, $\langle a,\S(\gamma)\rangle\in\gamma$.
\end{proof}

Here is the main result of this section.
\begin{thm}
\label{thm:espectro global para RCD almost univ}Suppose that $\mathcal{Q}$
has a $(r+1)$-ary NU term and that $\mathcal{Q}_{\mathrm{RSI}}$
is almost universal. If $\mathbf{A}\in\mathcal{Q}$ is RCD, then 
\[
\cmic_{\mathcal{Q}}(\mathbf{A})\cup\{\theta\in\con_{\mathcal{Q}}(\mathbf{A}):\left|\cov(\theta)\right|\in[2,r]\text{ and }\cov(\theta)\text{ is not a CRS}\}
\]
is a global spectrum of $\mathbf{A}$.
\end{thm}

\begin{proof}
Put $\Gamma:=\cmic_{\mathcal{Q}}(\mathbf{A})$; Lemma~\ref{lem:compacidad 1}
implies that 
\[
\Gamma\cup\{\nabla\}\text{ is }\ted\text{-closed and }\Gamma\text{ is }\te\text{-compact.}
\]
Let $\Sigma$ be the set of congruences in the statement and let $\mathsf{S}$
be a $\tau_{\mathrm{e}}$-system on $\Sigma$. By (1) of Lemma~\ref{lem:CMI alcanza}
we know that 
\[
\S\text{ is a }\sqcup\text{-system.}
\]
These facts are used several times throughout the proof.

A finite subset $\Omega\sub\Gamma$ is called an \emph{obstacle} if:
\begin{itemize}
\item $\S|_{\Omega}$ is unsolvable;
\item $\S|_{\Omega_{0}}$ is solvable for every $\Omega_{0}\varsubsetneq\Omega$.
\end{itemize}
Given a positive integer $k$ let $\mathcal{O}_{k}$ be the set of
all obstacles of size $k$. Of course $\mathcal{O}_{1}=\emptyset$.

Fix an arbitrary positive integer $k\geq2$.

\noindent\begin{customclaim}{I}If $\Lambda,\Omega\in\mathcal{O}_{k}$
and $\Lambda\mins\Omega$, then for every $\lambda\in\Lambda$ there
is a unique $\gamma_{\lambda}\in\Omega$ satisfying $\lambda\sub\gamma_{\lambda}$.\end{customclaim}

\noindent Note that $\Omega$ cannot be minorized by a proper subset
of $\Lambda$, as otherwise Lemma \ref{lem:minos} would render $\S|_{\Omega}$
solvable. Hence, as $\Lambda$ and $\Omega$ have the same finite
size, the claim follows.

\noindent\begin{customclaim}{II}If $\mathcal{O}_{k}$ is nonempty,
then it contains a $\mins$-maximal element.\end{customclaim}

\noindent Let $C\sub\mathcal{O}_{k}$ be a nonempty $\mins$-chain
and fix $\Lambda\in C$. Put $C_{\Lambda}:=\{\Omega\in C:\Lambda\mins\Omega\}$
and note that, by Claim I, for each $\Omega\in C_{\Lambda}$ and each
$\lambda\in\Lambda$ there is a unique $\lambda{}^{\Omega}\in\Omega$
satisfying $\lambda\sub\lambda^{\Omega}$. By the same claim, for
each $\lambda\in\Lambda$
\[
C_{\lambda}:=\{\lambda^{\Omega}:\Omega\in C_{\Lambda}\}
\]
 is a chain with respect to inclusion. So, for each $\lambda\in\Lambda$
\[
\lambda^{*}:=\bigcup C_{\lambda}
\]
is a $\mathcal{Q}$-congruence of $\mathbf{A}$. Define 
\[
\Lambda^{*}:=\{\lambda^{*}:\lambda\in\Lambda\},
\]
and note that, by Lemma \ref{lem:compacidad 2} we have 
\begin{equation}
\Lambda^{*}\sub\Gamma\cup\{\nabla\}.\label{eq: cmi o nabla}
\end{equation}

\noindent We claim that
\begin{equation}
\text{no proper subset of }\Lambda\text{ minorizes }\Lambda^{*}.\label{eq: ningun sub propio de Lambda min Lambda*}
\end{equation}
Suppose, aiming at a contradiction, that (\ref{eq: ningun sub propio de Lambda min Lambda*})
does not hold. Then, there is $\nu\in\Lambda$ such that $\Lambda\setminus\{\nu\}$
minorizes $\Lambda^{*}$. In particular, there is $\rho\in\Lambda\setminus\{\nu\}$
such that $\rho\sub\nu^{*}$. Since $\Lambda$ is an obstacle, there
is a solution $a\in A$ for $\S|_{\Lambda\setminus\{\nu\}}$. That
is, 
\[
\langle a,\S(\lambda)\rangle\in\lambda\text{ for all }\lambda\in\Lambda\setminus\{\nu\}.
\]
So, as $\langle\S(\rho),\S(\nu)\rangle\in\rho\sqcup\nu\sub\nu^{*}$
and $\langle a,\S(\rho)\rangle\in\nu^{*}$, we have 
\[
\langle a,\S(\nu)\rangle\in\nu^{*}.
\]
Then, by Claim I and the fact that $C_{\Lambda}$ is a $\mins$-chain,
there is $\mathsf{\Omega}\in C_{\Lambda}$ such that 
\begin{equation}
\langle a,\S(\lambda)\rangle\in\lambda^{\Omega}\text{ for all }\lambda\in\Lambda.\label{eq:paso 1}
\end{equation}

\noindent Hence, since 
\begin{equation}
\langle\S(\lambda),\S(\lambda^{\Omega})\rangle\in\lambda\sqcup\lambda^{\Omega}=\lambda^{\Omega}\text{ for all }\lambda\in\Lambda,\label{eq:paso 2}
\end{equation}
we have 
\begin{equation}
\langle a,\S(\lambda^{\Omega})\rangle\in\lambda^{\Omega}\text{ for all }\lambda\in\Lambda.\label{eq:paso3}
\end{equation}
This contradicts the fact that $\S|_{\Omega}$ is unsolvable.

It follows from (\ref{eq: ningun sub propio de Lambda min Lambda*})
that $\nabla\notin\Lambda^{*}$, which by (\ref{eq: cmi o nabla})
implies that $\Lambda^{*}\sub\Gamma$. Another consequence of (\ref{eq: ningun sub propio de Lambda min Lambda*})
is that $\left|\Lambda^{*}\right|=k$.

The argument used to prove (\ref{eq: ningun sub propio de Lambda min Lambda*})
also shows that 
\begin{equation}
\S|_{\Lambda^{*}}\text{ is not solvable}.\label{eq: T|Lambda* es insoluble}
\end{equation}
 Indeed, if there is $a\in A$ satisfying 
\[
\langle a,\S(\lambda^{*})\rangle\in\lambda^{*}\text{ for all }\lambda\in\Lambda,
\]
then, since $\langle\S(\lambda),\S(\lambda^{*})\rangle\in\lambda\sqcup\lambda^{*}=\lambda^{*}$,
we have 
\[
\langle a,\S(\lambda)\rangle\in\lambda^{*}\text{ for all }\lambda\in\Lambda.
\]
Now, taking a large enough $\Omega\in C_{\Lambda}$, steps (\ref{eq:paso 1})-(\ref{eq:paso3})
carry through and yield the same contradiction.

It remains to see that every proper subsystem of $\S|_{\Lambda^{*}}$
is solvable. Let $\Psi\varsubsetneq\Lambda^{*}$ and put $\Lambda_{\Psi}:=\{\lambda\in\Lambda:\lambda^{*}\in\Psi\}$.
Since $|\Lambda^{*}|=|\Lambda|$, we have $\Lambda_{\Psi}\varsubsetneq\Lambda$.
Thus, as $\Lambda$ is an obstacle, $\S|_{\Lambda_{\Psi}}$ has a
solution $a\in A$. If $\lambda^{*}\in\Psi$, then $\lambda\in\Lambda_{\Psi}$,
and hence $\langle a,\S(\lambda)\rangle\in\lambda\sub\lambda^{*}$.
Also, $\langle\S(\lambda),\S(\lambda^{*})\rangle\in\lambda\sqcup\lambda^{*}=\lambda^{*}$.
Therefore $\langle a,\S(\lambda^{*})\rangle\in\lambda^{*}$ for every
$\lambda^{*}\in\Psi$, and so $a$ is a solution for $\S|_{\Psi}$.

Collecting the facts obtained we can conclude that $\Lambda^{*}\in\mathcal{O}_{k}$.
As $\Lambda^{*}$ is an upper bound of $C$ with respect to the $\mins$
ordering, by Zorn's lemma, there is a $\mins$-maximal element in
$\mathcal{O}_{k}$.

\noindent\begin{customclaim}{III}If $k\leq r$ and $\mathcal{O}_{k}$
is nonempty, then $k<r$ and there is $\ell\in[k+1,r]$ such that
$\mathcal{O}_{\ell}$ is nonempty.\end{customclaim}

\noindent Suppose $\mathcal{O}_{k}$ is nonempty. Then, by Claim II,
there is a $\mins$-maximal element $\Omega$ in $\mathcal{O}_{k}$.
Note that $\bigcap\Omega\notin\Sigma$, since otherwise $\mathsf{S}(\bigcap\Omega)$
would be a solution for $\mathsf{S}|_{\Omega}$. Also, since $\Omega$
is an obstacle, its members are pairwise incomparable. So, Lemma~\ref{lem:atomos de Con A inducido por un conjunto finito de CMIs}
in combination with Lemma~\ref{lem:correspondencia} produce 
\[
2\leq|\cov(\bigcap\Omega)|\leq|\Omega|=k\leq r.
\]
Hence, in view of the definition of $\Sigma$, the set $\cov(\bigcap\Omega)$
must be a CRS. Thus, as $\mathsf{S}|_{\Omega}$ is a $\sqcup$-system,
Lemma~\ref{lem:almost solvable} produces $\delta\in\Omega$ such
that the system $\S_{\delta}$ on $(\Omega\setminus\{\delta\})\cup\{\hat{\delta}\}$
defined by 
\begin{align*}
\S_{\delta}(\hat{\delta}):= & \S(\delta),\\
\S_{\delta}(\gamma):= & \S(\gamma)\text{ for }\gamma\in(\Omega\setminus\{\delta\}),
\end{align*}
is not solvable. Notice that $\hat{\delta}\neq\nabla$. Indeed, if
$\hat{\delta}=\nabla$, then any solution of $\S|_{\Omega\setminus\{\delta\}}$
would be a solution of $\S_{\delta}$, contradicting the choice of
$\delta$. Define 
\[
\mu:=\hat{\delta}\cap\bigcap(\Omega\setminus\{\delta\}).
\]
We show next, by way of contradiction, that $\mathsf{S}|_{[\mu)\cap\Gamma}$
is not solvable. Suppose $a$ is a solution for $\mathsf{S}|_{[\mu)\cap\Gamma}$.
Take $\lambda\in\Gamma$ such that $\hat{\delta}\sub\lambda$. Since
$\lambda\in[\mu)$, we have $\langle a,\S(\lambda)\rangle\in\lambda$,
and as $\S$ is an $\sqcup$-system, we have $\langle\S(\delta),\S(\lambda)\rangle\in\lambda$.
Hence, $\langle a,\S(\delta)\rangle\in\lambda$, and as $\lambda$
is an arbitrary CMI congruence containing $\hat{\delta}$, it follows
that $\langle a,\S(\delta)\rangle\in\hat{\delta}$. That is,
\begin{equation}
\langle a,\S_{\delta}(\hat{\delta})\rangle\in\hat{\delta}.\label{eq:sol para delta barra}
\end{equation}
Furthermore, since $\Omega\setminus\{\delta\}\sub[\mu)\cap\Gamma$,
we have $\langle a,\S(\gamma)\rangle\in\gamma$ for all $\gamma\in\Omega\setminus\{\delta\}$.
Equivalently,
\begin{equation}
\langle a,\S_{\delta}(\gamma)\rangle\in\gamma\text{ for all }\gamma\in\Omega\setminus\{\delta\}.\label{eq:sol para el resto}
\end{equation}

\noindent It follows from (\ref{eq:sol para delta barra}) and (\ref{eq:sol para el resto}),
that $a$ is a solution for $\S_{\delta}$, a contradiction.

\noindent Note that $[\mu)\cap\Gamma$ is $\te$-compact by item (3)
of Lemma~\ref{lem:compacidad 1}, and $[\mu)\cap\Gamma$ is clearly
CMI-saturated. So, as $\mathsf{S}|_{[\mu)\cap\Gamma}$ is not solvable,
Lemma~\ref{lem:S r-soluble sobre CMI alcanza} says that it is not
$r$-solvable. Let $\Lambda\sub[\mu)\cap\Gamma$ be a set of minimum
cardinality such that $\mathsf{S}|_{\Lambda}$ is not solvable. Hence
$\Lambda$ is an obstacle, and
\[
\left|\Lambda\right|\leq r.
\]
Note that, since $\Lambda\sub[\mu)$ and every congruence in $\Lambda$
is meet-prime, we have $\Omega\mins\Lambda.$ Moreover, $\Lambda$
cannot be minorized by a proper subset of $\Omega$, as otherwise
Lemma~\ref{lem:minos} would say that $\S|_{\Lambda}$ is solvable.
Hence, 
\[
\left|\Lambda\right|\geq\left|\Omega\right|=k.
\]
Next, observe that since $\mu\nsubseteq\delta$ and $\mu$ is contained
in every congruence in $\Lambda$, no congruence in $\Lambda$ can
be contained in $\delta$. Thus $\Omega\mins\Lambda$, but $\Lambda\ntrianglelefteq\Omega$;
that is, $\Omega$ strictly minorizes $\Lambda$. If $|\Lambda|=|\Omega|$,
then $\Lambda\in\mathcal{O}_{k}$, contradicting the $\mins$-maximality
of $\Omega$. Hence $|\Lambda|>|\Omega|=k$. Since also $|\Lambda|\leq r$,
taking $\ell:=|\Lambda|$ yields 
\[
\ell\in\{k+1,\ldots,r\}\quad\text{and}\quad\mathcal{O}_{\ell}\neq\emptyset.
\]

Observe that, since $k\geq2$ was arbitrary, Claim III holds for every
$k\geq2$.

\noindent\begin{customclaim}{IV}$\mathcal{O}_{n}$ is empty for
all $n\in[2,r]$.\end{customclaim}

\noindent Suppose, for the sake of contradiction, that 
\[
\{n\in[2,r]:\mathcal{O}_{n}\neq\emptyset\}
\]
is nonempty, and let $m$ be its largest member. By Claim III, $m<r$
and there is $\ell\in[m+1,r]$ such that $\mathcal{O}_{\ell}\neq\emptyset$,
contradicting the maximality of $m$. Hence the claim is proved.

To conclude, note first that, by (2) of Lemma~\ref{lem:CMI alcanza},
\[
\S\text{ is solvable}\Longleftrightarrow\S|_{\Gamma}\text{ is solvable}.
\]
Moreover, since $\Gamma$ is $\tau_{\mathrm{e}}$-compact, Theorem~\ref{thm:BP congruencial}
gives 
\[
\S|_{\Gamma}\text{ is solvable}\Longleftrightarrow\S|_{\Gamma}\text{ is }r\text{-solvable}.
\]
The latter condition must hold. Otherwise, there would be $\Omega\subseteq\Gamma$
with $|\Omega|\leq r$ such that $\S|_{\Omega}$ is not solvable.
Taking such an $\Omega$ of minimum cardinality, we would obtain an
obstacle of size at most $r$. Since $\mathcal{O}_{1}=\emptyset$,
this contradicts Claim IV. Thus $\S$ is solvable. Since $\S$ was
arbitrary, $\Sigma$ is $\te$-complete. Therefore Lemma~\ref{lem:global iff te-complete}
implies that $\Sigma$ is a global spectrum of $\mathbf{A}$.
\end{proof}

The preceding Theorem immediately yields an optimal global representation
result for the quasivarieties in question.
\begin{cor}
\label{cor:A RCD implica RGI-spectral}Let $\mathcal{Q}$ be a quasivariety
with a $(r+1)$-ary NU term, and such that $\mathcal{Q}_{\mathrm{RSI}}$
is almost universal. If $\mathbf{A}\in\mathcal{Q}$ is RCD, then $\mathbf{A}$
is RGI-spectral.
\end{cor}

\begin{proof}
Combine Theorem~\ref{thm:espectro global para RCD almost univ} with
item~(2) of Lemma~\ref{lem:carac RGI RCD para NU y clase universal de SI}.
\end{proof}

\begin{cor}
\label{cor:super chela con NU es RGI espectral}Suppose that $\mathcal{Q}$
is RCD, has a $(r+1)$-ary NU term, and $\mathcal{Q}_{\mathrm{RSI}}$
is almost universal. Then $\mathcal{Q}$ is RGI-spectral and its class
of RGI members is
\[
\mathcal{Q}_{\mathrm{RGI}}=\mathcal{Q}_{\mathrm{RSI}}\cup\left\{ \mathbf{F}\in\mathcal{Q}:\begin{array}{l}
\operatorname{Con}_{\mathcal{Q}}(\mathbf{F})\text{ is }n\text{-atomic for some }n\in[2,r],\\
\text{and its set of atoms is not a CRS}
\end{array}\right\} .
\]
\end{cor}

\begin{proof}
Immediate by Corollary~\ref{cor:A RCD implica RGI-spectral} and
item (1) of Lemma~\ref{lem:carac RGI RCD para NU y clase universal de SI}.
\end{proof}

\subsection{\label{subsec:contraejemplo}RFSIs almost universal is not enough}

The aim of this subsection is to show that the assumption ``$\mathcal{Q}_{\mathrm{RSI}}$
is almost universal'' cannot be weakened to ``$\mathcal{Q}_{\mathrm{RFSI}}$
is almost universal'' in Theorem~\ref{thm:espectro global para RCD almost univ}
and its corollaries. More precisely, we construct an RCD quasivariety
$\mathcal{B}$ of Boolean algebra expansions such that $\mathcal{B}_{\mathrm{RFSI}}$
is almost universal, but $\mathcal{B}$ is not RGI-spectral.

At the heart of the construction is an algebra $\mathbf{A}$ with
universe $A:=\{0,1\}^{\omega}$, whose congruence lattice is carefully
engineered. We start by defining the intended congruences of $\mathbf{A}$.
Put $\theta_{0}:=\nabla$ and for $n\geq1$ define $\theta_{n}$ by

\[
a\mathrel{\theta_{n}}b\Longleftrightarrow a|_{[0,n-1]}=b|_{[0,n-1]}.
\]
Also, fix a nonprincipal ultrafilter $\mathcal{U}$ on $\omega$ and
put
\[
\beta:=\{\langle a,b\rangle\in A^{2}:\mathrm{E}(a,b)\in\mathcal{U}\}.
\]
Note that:
\begin{itemize}
\item the $\theta_{n}$'s form an infinite downward chain with intersection
$\Delta$;
\item $\beta$ is a maximal equivalence relation incomparable with $\theta_{n}$
for $n\geq1$;
\item $(\theta_{k}\cap\beta)\vee\theta_{k+1}=(\theta_{k}\cap\beta)\circ\theta_{k+1}=\theta_{k}$.
\end{itemize}
Let $\mathbf{L}$ be the bounded sublattice of the lattice of equivalence
relations on $A$ generated by $\{\theta_{n}:n\geq1\}\cup\{\beta\}$.
It follows from the above observations that the universe of $\mathbf{L}$
is
\[
L=\{\theta_{n}:n\geq1\}\cup\{\theta_{n}\cap\beta:n\geq1\}\cup\{\beta,\Delta,\nabla\},
\]
and its ordering is fully described by the diagram in Figure~\ref{fig:ConA}.

\begin{figure}[h]
\centering
\begin{centering}
\includegraphics[scale=0.75]{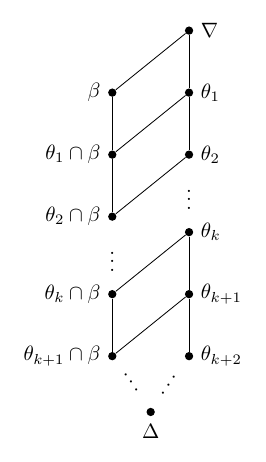}
\par\end{centering}
\caption{\label{fig:ConA}The congruence lattice of $\mathbf{A}$.}

\end{figure}

The next step is to endow $A$ with enough operations to ensure that
the resulting congruences are precisely the members of $L$. Let $\mathcal{F}$
be the set of all unary operations on $A$ compatible with the equivalence
relations in $\{\theta_{n}:n\geq1\}\cup\{\beta\}$, and let 
\[
\mathbf{A}_{0}:=\langle A,\mathcal{F}\rangle.
\]
For $a,b\in A$ let 
\[
\mathrm{D}(a,b):=\{i\in\omega:a_{i}\neq b_{i}\}
\]
 and define 
\[
\mu(a,b):=\begin{cases}
\min(\mathrm{D}(a,b)) & \text{if }a\neq b,\\
\infty & \text{if }a=b.
\end{cases}
\]

\begin{lem}
\label{lem:ConA0}The following hold:
\begin{enumerate}
\item If $a,b,c,d\in A$ are such that $\mu(a,b)\leq\mu(c,d)$ and either
$\mathrm{E}(c,d)\in\mathcal{U}$ or $\mathrm{D}(a,b)\in\mathcal{U}$,
then there is $f\in\mathcal{F}$ satisfying $f(a)=c$ and $f(b)=d$.
\item The congruence lattice of $\mathbf{A}_{0}$ is $\mathbf{L}$.
\end{enumerate}
\end{lem}

\begin{proof}
(1). If $a=b$ the claim holds trivially, so assume $a\neq b$. Put
$k:=\mu(a,b)$ and define $f:A\rightarrow A$ by
\[
f(x)_{i}:=\begin{cases}
c_{i} & \text{if }i\in\mathrm{E}(c,d),\\
x_{i} & \text{if }i\in\mathrm{D}(c,d)\cap\mathrm{D}(a,b)\cap\mathrm{E}(a,c),\\
1-x_{i} & \text{if }i\in\mathrm{D}(c,d)\cap\mathrm{D}(a,b)\cap\mathrm{D}(a,c),\\
c_{i} & \text{if }i\in\mathrm{D}(c,d)\cap\mathrm{E}(a,b)\text{ and }x_{k}=a_{k},\\
d_{i} & \text{if }i\in\mathrm{D}(c,d)\cap\mathrm{E}(a,b)\text{ and }x_{k}=b_{k}.
\end{cases}
\]
Clearly we have $f(a)=c$ and $f(b)=d$. It remains to show that $f$
is compatible with every $\theta_{n}$ and $\beta$. Fix $u,v\in A$
and note that 
\begin{equation}
\mathrm{E}(c,d)\cup(\mathrm{D}(c,d)\cap\mathrm{D}(a,b)\cap\mathrm{E}(u,v))\sub\mathrm{E}(f(u),f(v)).\label{eq:inc clave}
\end{equation}
Suppose $\langle u,v\rangle\in\beta$, that is $\mathrm{E}(u,v)\in\mathcal{U}$.
If $\mathrm{E}(c,d)\in\mathcal{U}$, then (\ref{eq:inc clave}) yields
$\mathrm{E}(f(u),f(v))\in\mathcal{U}$. Otherwise, by assumption we
have $\mathrm{D}(a,b)\cap\mathrm{D}(c,d)\in\mathcal{U}$, and it follows
from (\ref{eq:inc clave}) that $\mathrm{E}(f(u),f(v))\in\mathcal{U}$.
This proves that $f$ preserves $\beta$.

We show next that 
\begin{equation}
\mu(u,v)\leq\mu(f(u),f(v)).\label{eq:Lipschitz}
\end{equation}
Suppose first that $u_{k}=v_{k}$. Then the last two clauses in the
definition of $f$ take the same value on $u$ and $v$, while the
second and third clauses preserve coordinatewise equality and inequality.
Consequently, 
\[
\mathrm{D}(f(u),f(v))=\mathrm{D}(u,v)\cap\mathrm{D}(c,d)\cap\mathrm{D}(a,b)\subseteq\mathrm{D}(u,v),
\]
and hence (\ref{eq:Lipschitz}) holds.

Suppose now that $u_{k}\neq v_{k}$. Then 
\[
\mu(u,v)\leq k\leq\mu(c,d).
\]
Moreover, $f(u)$ and $f(v)$ agree at every coordinate in $\mathrm{E}(c,d)$,
and therefore 
\[
\mu(c,d)\leq\mu(f(u),f(v)).
\]
Thus (\ref{eq:Lipschitz}) holds in this case as well. Since (\ref{eq:Lipschitz})
is equivalent to the preservation of every $\theta_{n}$, we conclude
that $f\in\mathcal{F}$.

(2). By the shape of $\mathbf{L}$ it suffices to show that for all
$a\neq b$ in $A$ we have
\[
\cg^{\mathbf{A}_{0}}(a,b)=\begin{cases}
\theta_{\mu(a,b)}\cap\beta & \text{if }\langle a,b\rangle\in\beta,\\
\theta_{\mu(a,b)} & \text{if }\langle a,b\rangle\notin\beta.
\end{cases}
\]
Fix distinct $a,b\in A$. Since in either case $\langle a,b\rangle$
belongs to the congruence on the right-hand side, the inclusion from
left-to-right holds. Suppose first that $\langle a,b\rangle\in\beta$
and take $\langle c,d\rangle\in\theta_{\mu(a,b)}\cap\beta$. Then
$\mu(a,b)\leq\mu(c,d)$ and $\mathrm{E}(c,d)\in\mathcal{U}$. So (1)
produces $f\in\mathcal{F}$ with $f(a)=c$ and $f(b)=d$. It follows
that $\langle c,d\rangle\in\cg^{\mathbf{A}_{0}}(a,b)$. The case $\langle a,b\rangle\notin\beta$
is analogous.
\end{proof}

The next step is to add operations to $\mathbf{A}_{0}$ without disturbing
its congruence lattice and so that the resulting algebra generates
a quasivariety with the desired properties. Let 
\[
\mathbf{A}:=\langle\{0,1\}^{\omega},\mathcal{F},\wedge,\vee,\neg,\bot,\top,q\rangle,
\]
where $\wedge,\vee,\neg,\bot,\top$ are the usual coordinatewise operations
on $\{0,1\}^{\omega}$ and $q$ is the quaternary operation given
by:
\[
q(a,b,c,d)_{i}:=\begin{cases}
1 & \text{if }i=\max(\mu(a,b),\mu(c,d)),\\
0 & \text{otherwise}.
\end{cases}
\]
Notice that if either $a=b$ or $c=d$, then $q(a,b,c,d)=\bot$.
\begin{lem}
The congruence lattice of $\mathbf{A}$ is $\mathbf{L}$.
\end{lem}

\begin{proof}
Since $\mathbf{A}$ is an expansion of $\mathbf{A}_{0}$, by item
(2) of Lemma~\ref{lem:ConA0} it suffices to show that the added
operations preserve the congruences in $L$. This is easy to check
for the Boolean operations. Moreover, as $\langle q(a,b,c,d),\bot\rangle\in\beta$
for all $a,b,c,d\in A$, we clearly have that $q$ preserves $\beta$.
It remains to show that $q$ preserves each $\theta_{n}$. Fix $n\geq1$
and suppose $\langle a,a'\rangle,\langle b,b'\rangle,\langle c,c'\rangle,\langle d,d'\rangle\in\theta_{n}$.
Let 
\[
\ell:=\mu(a,b),\quad\ell':=\mu(a',b'),\quad k:=\mu(c,d),\quad k':=\mu(c',d').
\]
Note that 
\begin{equation}
\ell<n\Rightarrow\ell=\ell'\text{ and }\ell\geq n\Rightarrow\ell'\geq n,\label{eq:eles y enes}
\end{equation}
and the same holds for $k$ and $k'$. Now put 
\[
m:=\max(\ell,k)\quad m':=\max(\ell',k').
\]
There are two cases.

If $m<n$, then~(\ref{eq:eles y enes}) yields $\ell=\ell'$ and
$k=k'$. Thus $m=m'$, which implies $q(a,b,c,d)=q(a',b',c',d')$.

If $m\geq n$, then~(\ref{eq:eles y enes}) says that either $\ell'\geq n$
or $k'\geq n$, and hence $m'\geq n$. It follows that 
\[
q(a,b,c,d)_{i}=0=q(a',b',c',d')_{i}\quad\text{for }i\in[0,n-1].
\]
\end{proof}

We can now introduce the quasivariety. For $n\geq1$ define

\[
\mathbf{A}_{n}:=\mathbf{A}/\theta_{n}\text{ and }\mathbf{B}:=\mathbf{A}/\beta,
\]
and let $\mathcal{B}$ be the quasivariety generated by $\{\mathbf{A}_{n}:n\geq1\}\cup\{\mathbf{B}\}$.
\begin{prop}
The following hold:
\begin{enumerate}
\item $\mathcal{B}$ is RCD, has a majority term and $\mathcal{B}_{\mathrm{RFSI}}$
is universal.
\item $\mathbf{A}\in\mathcal{B}$ and $\rgic_{\mathcal{B}}(\mathbf{A})=\{\theta_{n}:n\geq1\}\cup\{\beta\}$.
\item $\mathbf{A}$ is not RGI-spectral.
\end{enumerate}
\end{prop}

\begin{proof}
(1). Define $\mathcal{K}:=\{\mathbf{A}_{n}:n\geq1\}\cup\{\mathbf{B}\}$.
Let $\cp$ denote the symmetric difference operation. Let 
\[
p(x,y,z,w):=(x\cp y)\wedge(z\cp w).
\]
A direct computation gives 
\begin{equation}
\mathcal{K}\vDash(x=y\text{ or }z=w)\longleftrightarrow\bigl(q(x,y,z,w)=\bot\text{ and }p(x,y,z,w)=\bot\bigr).\label{eq:traduccion}
\end{equation}
By the relative Jónsson lemma, 
\[
\mathcal{B}_{\mathrm{RFSI}}\subseteq\mathbb{ISP}_{u}(\mathcal{K}).
\]
Since the condition displayed in (\ref{eq:traduccion}) is universal,
it holds throughout $\mathcal{B}_{\mathrm{RFSI}}$. Hence the pairs
\[
\bigl(q(x,y,z,w),\bot\bigr)\quad\text{and}\quad\bigl(p(x,y,z,w),\bot\bigr)
\]
witness EDPM for $\mathcal{B}$ by \cite[Theorem~2.3]{CzelakowskiDziobiak1990}.
In particular, $\mathcal{B}$ is RCD and $\mathcal{B}_{\mathrm{RFSI}}$
is universal. Finally, the term $M(x,y,z):=(x\vee y)\wedge(x\vee z)\wedge(y\vee z)$
is a majority term for $\mathcal{B}$.

(2). Let $\Sigma:=\{\theta_{n}:n\geq1\}$. Since $\bigcap\Sigma=\Delta$,
we have that $\mathbf{A}$ embeds in $\prod_{n\geq1}\mathbf{A}_{n}$,
and thus $\mathbf{A}\in\mathcal{B}$. Moreover, 
\[
\con_{\mathcal{B}}(\mathbf{A})=\con(\mathbf{A}).
\]
We claim that $\Sigma$ is $\te$-complete. Let $\S$ be a $\te$-system
on $\Sigma$ and $\mathcal{C}$ a $\te$-cover of $\Sigma$ compatible
with $\S$. For each $n\geq1$ choose $u_{n}\in\mathcal{C}$ with
$\theta_{n}\in u_{n}$ and fix a solution $a_{n}\in A$ for $\S|_{u_{n}}$.
Note that $\theta_{n}\in u_{n+1}$ and hence 
\[
\langle a_{n},a_{n+1}\rangle\in\theta_{n}\quad\text{for all }n\geq1.
\]
Equivalently, the sequences $a_{n}$ and $a_{n+1}$ agree on their
first $n$ coordinates. Hence, there is a unique sequence $a\in A$
satisfying $a|_{[0,n-1]}=a_{n}|_{[0,n-1]}$ for all $n\geq1$. It
is straightforward to check that $a$ is a solution for $\S$. 

It follows that $\Sigma$ is $\te$-complete, and Lemma~\ref{lem:carac congruencial de RGI}
says that $\mathbf{A}\notin\mathcal{B}_{\mathrm{RGI}}$; that is $\Delta\notin\rgic_{\mathcal{B}}(\mathbf{A})$.
As the congruences $\theta_{n}$ and $\beta$ permute for each $n\geq1$,
we have that $\mathbf{A}/(\theta_{n}\cap\beta)\cong\mathbf{A}_{n}\times\mathbf{B}$;
hence $\theta_{n}\cap\beta\notin\rgic_{\mathcal{B}}(\mathbf{A})$.
So, since every CMI congruence is RGI, we have $\Sigma\cup\{\beta\}=\rgic_{\mathcal{B}}(\mathbf{A}).$

(3). We show that $\Gamma:=\{\theta_{n}:n\geq1\}\cup\{\beta\}$ is
not $\te$-complete. Let $\S:\Gamma\rightarrow A$ be the system given
by $\S(\beta):=\top$ and $\S(\theta_{n}):=0^{n}1^{\omega}$ for all
$n\geq1$. We claim that $\S$ is a $\te$-system. For each $n\geq1$
define $u_{n}:=\mathrm{e}(0^{n}1^{\omega},\bot)$. Note that
\[
\theta_{n}\in u_{n}\text{ and }\beta\notin u_{n}\quad\text{for all }n\geq1.
\]
Put $v:=\mathrm{e}(10^{\omega},\bot)$ and observe that
\[
\beta\in v\text{ and }\theta_{n}\notin v\quad\text{for all }n\geq1.
\]
It follows that $\mathcal{C}:=\{u_{n}:n\geq1\}\cup\{v\}$ is a $\te$-cover
of $\Gamma$. Moreover, it is compatible with $\S$ since $\top$
is a solution for $\S|_{v}$ and $0^{n}1^{\omega}$ solves $\S|_{u_{n}}$
for each $n\geq1$. 

Finally, note that $\S$ is not solvable, as the only solution for
$\S|_{\{\theta_{n}:n\geq1\}}$ is $\bot$ and $\langle\bot,\top\rangle\notin\beta$.
\end{proof}

\subsection{Global representations with bounded width factors require a NU term}

In this section we provide a converse to Corollary \ref{cor:super chela con NU es RGI espectral}.
\begin{lem}
\label{lem:varias cosas implican NU}Let $r\geq2$. Suppose that the
$\mathcal{Q}$-free algebra on $r+1$ generators is RCD, and that
every $(r+1)$-generated algebra in $\mathcal{Q}$ has a global representation
whose factors have relative subdirect width at most $r$. Then $\mathcal{Q}$
has a $(r+1)$-ary NU term.
\end{lem}

\begin{proof}
Let $\mathbf{F}$ be the $\mathcal{Q}$-free algebra freely generated
by $x_{1},\ldots,x_{r+1}$. For each $k\in[1,r+1]$ define 
\[
\theta_{k}:=\bigsqcup\{\cg^{\mathbf{F}}_{\mathcal{Q}}(x_{i},x_{j}):i,j\neq k\},
\]
and put
\[
\delta:=\bigcap^{r+1}_{k=1}\theta_{k}.
\]
By assumption, the algebra $\mathbf{F}/\delta$ has a global spectrum
$\tilde{\Sigma}\sub\cmic_{\mathcal{Q}}(\mathbf{F}/\delta)^{(r)}$.
So $\tilde{\Sigma}$ is $\te$-complete by Lemma~\ref{lem:global iff te-complete},
and Lemma~\ref{lem:correspondencia} says that 
\[
\Sigma:=\{\theta\in\con_{\mathcal{Q}}(\mathbf{F}):\delta\sub\theta\text{ and }\theta/\delta\in\tilde{\Sigma}\}
\]
is $\te$-complete. Moreover, we have that $\bigcap\Sigma=\delta$
and $\Sigma\sub\cmic_{\mathcal{Q}}(\mathbf{F})^{(r)}$.

Next, for $i\in[1,r+1]$ define 
\[
\delta_{i}:=\bigcap_{k\neq i}\theta_{k}.
\]
We claim that 
\begin{equation}
\langle x_{i},x_{j}\rangle\in\delta_{i}\sqcup\delta_{j}\text{ for all }i,j.\label{eq:supremos de deltas}
\end{equation}
The case $i=j$ is trivial, so suppose that $i\neq j$. By distributivity,

\[
\delta_{i}\sqcup\delta_{j}=\bigcap_{k\neq i,\ell\neq j}\theta_{k}\sqcup\theta_{\ell}.
\]
Fix $k\neq i$ and $\ell\neq j$. If $k\neq\ell$, choose $m\in[1,r+1]\setminus\{k,\ell\}$.
Then $\langle x_{i},x_{m}\rangle\in\theta_{k}$ and $\langle x_{m},x_{j}\rangle\in\theta_{\ell}$,
which implies $\langle x_{i},x_{j}\rangle\in\theta_{k}\sqcup\theta_{\ell}$.
If $k=\ell$, then $k\notin\{i,j\}$ and hence $\langle x_{i},x_{j}\rangle\in\theta_{k}$.
So, 
\[
\langle x_{i},x_{j}\rangle\in\bigcap_{k\neq i,\ell\neq j}\theta_{k}\sqcup\theta_{\ell},
\]
 which proves (\ref{eq:supremos de deltas}).

We may assume that $\mathcal{Q}$ is nontrivial. We next observe that
the congruences 
\begin{equation}
\delta_{1},\ldots,\delta_{r+1}\text{ are pairwise incomparable.}\label{eq:las deltas son distintas}
\end{equation}
Indeed, suppose that $i\neq j$ and $\delta_{j}\sub\delta_{i}$. By~
(\ref{eq:supremos de deltas}) we have $\langle x_{i},x_{j}\rangle\in\delta_{i}\sqcup\delta_{j}=\delta_{i}\sub\theta_{j}.$
Take a nontrivial algebra $\mathbf{A}\in\mathcal{Q}$ and distinct
elements $a,b\in A$. By freeness, there is a homomorphism $h\colon\mathbf{F}\to\mathbf{A}$
such that $h(x_{j})=b$ and $h(x_{k})=a$ for every $k\neq j$. Then
$\theta_{j}\sub\ker h$, whereas $\langle x_{i},x_{j}\rangle\notin\ker h$,
a contradiction.

The next step is to prove that
\begin{equation}
\{\delta_{i}:i\in[1,r+1]\}\mins\Sigma.\label{eq:los deltas minorizan Sigma}
\end{equation}
Let $\theta\in\Sigma$ and take $\Theta\sub\cmic_{\mathcal{Q}}(\mathbf{F})$,
with $|\Theta|\leq r$, and such that $\theta=\bigcap\Theta$. Since
\[
\bigcap^{r+1}_{k=1}\theta_{k}\sub\theta=\bigcap\Theta,
\]
 and the congruences in $\Theta$ are meet-prime, it follows that
for each $\gamma\in\Theta$ there is $k_{\gamma}\in[1,r+1]$ satisfying
$\theta_{k_{\gamma}}\sub\gamma$. So, as $|\Theta|\leq r$, there
is 
\[
i\in[1,r+1]\setminus\{k_{\gamma}:\gamma\in\Theta\},
\]
and hence 
\[
\delta_{i}\sub\bigcap\Theta=\theta,
\]
as desired.

Put 
\[
\Omega:=\{\delta_{i}:i\in[1,r+1]\}\quad\text{and}\quad\widetilde{\Omega}:=\{\delta_{i}/\delta:i\in[1,r+1]\}.
\]
By~(\ref{eq:los deltas minorizan Sigma}), the set $\widetilde{\Omega}$
minorizes $\widetilde{\Sigma}$. So, Lemma~\ref{lem: minoriza esp glo rcd implica crs}
implies that $\widetilde{\Omega}$ is a CRS. Therefore, Lemma~\ref{lem:correspondencia}
shows that $\Omega$ is a CRS.

The map $\mathsf{S}:\Omega\rightarrow F$, given by $\mathsf{S}(\delta_{i}):=x_{i}$,
is well-defined by (\ref{eq:las deltas son distintas}) and it is
a $\sqcup$-system by (\ref{eq:supremos de deltas}). Hence, there
is a solution $s\in F$ for $\mathsf{S}$. Since $\mathbf{F}$ is
generated by $x_{1},\ldots,x_{r+1}$, there is a $(r+1)$-ary term
$M$ such that 
\[
s=M^{\mathbf{F}}(x_{1},\ldots,x_{r+1}).
\]
To conclude, we show that $M$ is a NU term for $\mathcal{Q}$. Let
$\mathbf{A}\in\mathcal{Q}$ and fix $a,b\in A$. By symmetry, it suffices
to prove that $a=M^{\mathbf{A}}(b,a,\ldots,a)$. Let $h:\mathbf{F}\rightarrow\mathbf{A}$
be a homomorphism extending the map $x_{1},x_{2},\ldots,x_{r+1}\mapsto b,a,\ldots,a$.
Note that 
\[
\langle s,x_{2}\rangle\in\delta_{2}\sub\theta_{1}\sub\ker h,
\]
and hence 
\[
a=h(x_{2})=M^{\mathbf{A}}(h(x_{1}),h(x_{2}),\ldots,h(x_{r+1}))=M^{\mathbf{A}}(b,a,\ldots,a).
\]
\end{proof}

Combining the preceding lemma with the forward representation theorem,
we obtain the following characterization.
\begin{thm}
\label{thm:global ancho finito sii NU}Let $r\geq2$. Suppose that
$\mathcal{Q}$ is RCD and that $\mathcal{Q}_{\mathrm{RSI}}$ is almost
universal. The following are equivalent:
\begin{enumerate}
\item $\mathcal{Q}$ has an $(r+1)$-ary NU term.
\item Every algebra in $\mathcal{Q}$ is isomorphic to a global subdirect
product whose factors have relative subdirect width at most $r$.
\item Every $(r+1)$-generated algebra in $\mathcal{Q}$ is isomorphic to
a global subdirect product whose factors have relative subdirect width
at most $r$.
\end{enumerate}
\end{thm}

\begin{proof}
The implication (1)$\Rightarrow$(2) follows from Corollary~\ref{cor:super chela con NU es RGI espectral},
(2)$\Rightarrow$(3) is immediate, and Lemma~\ref{lem:varias cosas implican NU}
yields (3)$\Rightarrow$(1).
\end{proof}

\subsection{Congruence permutability}

As discussed in \cite{vaggione_2022}, global representation results
have proven to be valuable tools for characterizing RCP algebras,
and Corollary~\ref{cor:super chela con NU es RGI espectral} is no
exception. To provide some context, recall Nachbin's classical characterization
\cite{nachbin_1947}: a distributive lattice is CP if and only if
it omits the three-element chain as a quotient. Since the two- and
three-element chains are precisely the globally indecomposable distributive
lattices, this can be rephrased by saying that a distributive lattice
is CP if and only if all its globally indecomposable quotients are
subdirectly irreducible. Similar results have been obtained for lattice
expansions and lattice-ordered structures \cite{GramagliaVaggione1996,GramagliaVaggione1997},
implication algebras \cite{Campercholi2010Implication}, and $\mathrm{MS}$-algebras
\cite{CampercholiVaggione2007MS}. A particularly relevant recent
example is the finite RCD case treated in \cite[Theorem 5.2]{CampercholiVaggioneZigaran:RCDGlobal}.
\begin{thm}
\label{thm:tipo Nachbin para superchela}Let $r\geq2$. Suppose that
$\mathcal{Q}$ has a $(r+1)$-ary NU term and that $\mathcal{Q}_{\mathrm{RSI}}$
is almost universal. For $\mathbf{A}\in\mathcal{Q}$ the following
are equivalent:
\begin{enumerate}
\item $\mathbf{A}$ is RCP.
\item $\mathbf{A}$ is RCD and for all $\theta,\delta\in\cmic_{\mathcal{Q}}(\mathbf{A})^{(r)}$
we have $\theta\sqcup\delta=\theta\circ\delta$.
\item Every subset of $\cmic_{\mathcal{Q}}(\mathbf{A})$ with at most $r$
elements is a CRS.
\item Every $\te$-compact CMI-saturated subset of $\con_{\mathcal{Q}}(\mathbf{A})$
is $\te$-complete.
\item Every finite subset of $\con_{\mathcal{Q}}(\mathbf{A})$ is a CRS.
\item $\mathbf{A}$ is RCD and $\rgic_{\mathcal{Q}}(\mathbf{A})=\cmic_{\mathcal{Q}}(\mathbf{A})$
or, equivalently, every RGI homomorphic image of $\mathbf{A}$ is
RSI.
\end{enumerate}
\end{thm}

\begin{proof}
Put $\Gamma:=\cmic_{\mathcal{Q}}(\mathbf{A})$.

(1)$\Rightarrow$(2). As $\mathbf{A}$ is RCP, we have that $\con_{\mathcal{Q}}(\mathbf{A})$
is a sublattice of $\con(\mathbf{A})$, which is distributive because
$\mathbf{A}$ has a NU-term.

(2)$\Rightarrow$(3). We prove by induction that: if $k\leq r$, then
every $k$-element subset of $\Gamma$ is a CRS. Let $\gamma_{1},\ldots,\gamma_{k}\in\Gamma$.
If $k=1$ the claim is trivial. Assume $k>1$ and fix a $\sqcup$-system
\begin{equation}
\gamma_{1},\ldots,\gamma_{k}\mapsto a_{1},\ldots,a_{k}\label{eq:sistema}
\end{equation}
By the induction hypothesis there is a solution $b$ for the system
\[
\gamma_{2},\ldots,\gamma_{k}\mapsto a_{2},\ldots,a_{k}.
\]
Then, $\langle a_{1},b\rangle\in\gamma_{1}\sqcup\gamma_{i}$ for all
$i\geq2$, and by distributivity we have 
\[
\langle a_{1},b\rangle\in\gamma_{1}\sqcup\bigcap_{i\geq2}\gamma_{i}.
\]
In view of (2), there is $a\in A$ such that $\langle a_{1},a\rangle\in\gamma_{1}$
and $\langle a,b\rangle\in\bigcap_{i\geq2}\gamma_{i}.$ It is easy
to see that $a$ is a solution for (\ref{eq:sistema}).

(3)$\Rightarrow$(4). Let $\Sigma$ be a $\te$-compact CMI-saturated
subset of $\con_{\mathcal{Q}}(\mathbf{A})$, and let $\mathsf{S}$
be a $\te$-system on $\Sigma$. By Lemma \ref{lem:CMI alcanza},
$\S$ is a $\sqcup$-system, so (3) says that $\S|_{[\Sigma)\cap\Gamma}$
is $r$-solvable. Hence $\mathsf{S}$ is solvable by Lemma~\ref{lem:S r-soluble sobre CMI alcanza},
which proves that $\Sigma$ is $\te$-complete.

(4)$\Rightarrow$(5). Fix a finite $\Omega\sub\con_{\mathcal{Q}}(\mathbf{A})$
and put $\Sigma:=[\Omega)\cap\Gamma$. By item (3) of Lemma~\ref{lem:compacidad 1},
$\Sigma$ is $\te$-compact, and it is clearly CMI-saturated. Hence,
$\Sigma$ is $\te$-complete by (4).

Let $\mathsf{S}$ be a $\sqcup$-system on $\Omega$. By Lemma~\ref{lem:extender sistema para arriba},
it extends to a $\sqcup$-system $\mathsf{S}^{+}$ on $\Sigma$ such
that $\mathsf{S}\mins\mathsf{S}^{+}$. Since $\mathsf{S}$ is finite,
Lemma~\ref{lem:min por sistema finito implica te sistema} says that
$\mathsf{S}^{+}$ is a $\te$-system. Thus, it has a solution $a\in A$.

Fix $\mu\in\Omega$. If $\gamma\in[\mu)\cap\Gamma$, then $\gamma\in\Sigma$,
and hence 
\[
\langle a,\mathsf{S}^{+}(\gamma)\rangle,\langle\mathsf{S}(\mu),\mathsf{S}^{+}(\gamma)\rangle\in\gamma.
\]
Therefore, 
\[
\langle a,\mathsf{S}(\mu)\rangle\in\bigcap([\mu)\cap\Gamma)=\mu.
\]
Thus, $a$ is a solution for $\mathsf{S}$, and $\Omega$ is a CRS.

(5)$\Rightarrow$(6). Lemma~\ref{lem: relativamente aritmetica sii todo finito es CRS}
implies that $\mathbf{A}$ is RCD, so it follows from condition (5)
and item (2) of Lemma~\ref{lem:carac RGI RCD para NU y clase universal de SI}
that $\rgic_{\mathcal{Q}}(\mathbf{A})=\cmic_{\mathcal{Q}}(\mathbf{A})$.

(6)$\Rightarrow$(1). Let $\theta$ and $\delta$ be $\mathcal{Q}$-congruences
of $\mathbf{A}$. Observe that by Lemma~\ref{lem:correspondencia}
we may assume $\theta\cap\delta=\Delta$. Since every congruence in
$\Gamma$ is meet-prime, Lemma~\ref{lem:completamente meet-prime}
says that $\{\theta,\delta\}$ minorizes $\Gamma$. By (6) and Theorem~\ref{thm:espectro global para RCD almost univ},
$\Gamma$ is $\te$-complete. Moreover, $\bigcap\Gamma=\Delta$. Therefore,
Lemma~\ref{lem: minoriza esp glo rcd implica crs} implies that $\{\theta,\delta\}$
is a CRS; that is, $\theta\sqcup\delta=\theta\circ\delta$.
\end{proof}

In particular, the preceding theorem recovers Nachbin's theorem for
bounded distributive lattices. It also yields the corresponding characterizations
for lattice expansions and lattice-ordered structures \cite{GramagliaVaggione1996,GramagliaVaggione1997}
and $\mathrm{MS}$-algebras \cite{CampercholiVaggione2007MS}.

\subsection{The semisimple case}

\subsubsection{Filtral quasivarieties}

The quasivariety $\mathcal{Q}$ is \emph{filtral} \cite{CamRaf17-RelCongForm}
if it is RCD and relatively semisimple, and $\mathcal{Q}_{\mathrm{RS}}$
is almost universal.
\begin{cor}
\label{cor:filtral implica RGI spectral}If $\mathcal{Q}$ is filtral
and has a $(r+1)$-ary NU term, then $\mathcal{Q}$ is RGI-spectral
and 
\[
\mathcal{Q}_{\mathrm{RGI}}=\mathcal{Q}_{\mathrm{RS}}\cup\{\mathbf{F}\in\mathcal{Q}:\mathbf{F}\text{ is a relative }g\text{-algebra and }\sdw_{\mathcal{Q}}(\mathbf{F})\leq r\}.
\]
\end{cor}

\begin{proof}
Since $\mathcal{Q}$ is relatively semisimple, we have $\mathcal{Q}_{\mathrm{RSI}}=\mathcal{Q}_{\mathrm{RS}}$.
Hence, Corollary~\ref{cor:super chela con NU es RGI espectral} implies
that $\mathcal{Q}$ is RGI-spectral. Let $\mathbf{F}\in\mathcal{Q}_{\mathrm{RGI}}\setminus\mathcal{Q}_{\mathrm{RS}}$.
By the same corollary and Lemma~\ref{lem:equivalencias RCD ancho finito},
there is an irredundant subdirect representation $\mathbf{F}\cong\mathbf{G}\leq_{\mathrm{isd}}\mathbf{S}_{1}\times\cdots\times\mathbf{S}_{n}$,
where $n=\sdw_{\mathcal{Q}}(\mathbf{F})\in[2,r]$ and $\mathbf{S}_{1},\ldots,\mathbf{S}_{n}\in\mathcal{Q}_{\mathrm{RS}}$.
Since $\mathbf{F}$ is RGI, Lemma~\ref{lem:RGI semisimple RCD} implies
that some tuple is almost in $G$. Thus, $\mathbf{F}$ is a relative
$g$-algebra. The converse is provided by Corollary~\ref{cor:toda g-algebra es RGI}.
\end{proof}

\subsubsection{Dual discriminator varieties}

A variety $\mathcal{V}$ is a \emph{dual discriminator} variety if
there are a class $\mathcal{S}\subseteq\mathcal{V}$ and a term $q(x,y,z)$
such that: 
\begin{itemize}
\item $\mathcal{V}$ is generated by $\mathcal{S}$;
\item For all $\mathbf{S}\in\mathcal{S}$ the interpretation $q^{\mathbf{S}}$
is the \emph{dual discriminator function}. That is, 
\[
q^{\mathbf{S}}(a,b,c)=\begin{cases}
a & \text{if }a=b,\\
c & \text{otherwise,}
\end{cases}
\]
for all $a,b,c\in S$.
\end{itemize}
Note that $q$ is a majority term for $\mathcal{V}$, and hence $\mathcal{V}$
is congruence-distributive.

As shown in \cite{Pixley_Fried_Dual_Disc}, every dual discriminator
variety is semisimple, and its simple members constitute an almost
universal class. In particular, dual discriminator varieties are filtral.

The dual discriminator severely constrains the form of irredundant
subdirect products with simple factors.
\begin{lem}
\label{lem:sub square}Let $\mathcal{V}$ be a dual discriminator
variety and suppose $\mathbf{C}\leq_{\mathrm{isd}}\mathbf{S}\times\mathbf{S}'$
for some $\mathbf{S},\mathbf{S}'\in\mathcal{V}_{\mathrm{S}}$. Then
exactly one of the following holds:
\begin{enumerate}
\item $\mathbf{C}=\mathbf{S}\times\mathbf{S}'$;
\item There exist $s\in S$ and $s'\in S'$ such that 
\[
C=(\{s\}\times S')\cup(S\times\{s'\}).
\]
\end{enumerate}
\end{lem}

\begin{proof}
This follows from the proof of \cite[Theorem 2.4]{Pixley_Fried_Dual_Disc}. 
\end{proof}

If $\mathbf{C}$ is as in (2) of Lemma \ref{lem:sub square} we say
that $\mathbf{C}$ is a \emph{simple cross}.
\begin{cor}
\label{cor:global dd}If $\mathcal{V}$ is a dual discriminator variety,
then $\mathcal{V}$ is GI-spectral\footnote{When $\mathcal{V}$ is a variety $\con_{\mathcal{V}}(\mathbf{A})=\con(\mathbf{A})$
for every $\mathbf{A}\in\mathcal{V}$. So in that case we omit the
adjective ``relative'' and the prefix $\mathrm{R}$ from the corresponding
terminology and notation.} and
\[
\mathcal{V}_{\mathrm{GI}}=\mathcal{V}_{\mathrm{S}}\cup\{\mathbf{C}\in\mathcal{V}:\mathbf{C}\text{ is isomorphic to a simple cross}\}.
\]
\end{cor}

\begin{proof}
Combine Corollary~\ref{cor:filtral implica RGI spectral} and Lemma~\ref{lem:sub square}.
\end{proof}

\subsubsection{Boolean representations in discriminator varieties}

A variety $\mathcal{V}$ is a \emph{discriminator variety} if there
are a class $\mathcal{S}\sub\mathcal{V}$ and a term $t(x,y,z)$ such
that:
\begin{itemize}
\item $\mathcal{V}$ is generated by $\mathcal{S}$;
\item For all $\mathbf{S}\in\mathcal{S}$ the interpretation $t^{\mathbf{S}}$
is the \emph{ternary discriminator function}. That is, 
\[
t^{\mathbf{S}}(a,b,c)=\begin{cases}
c & \text{if }a=b,\\
a & \text{otherwise,}
\end{cases}
\]
for all $a,b,c\in S$.
\end{itemize}
Note that in this case the interpretation of $q(x,y,z):=t(x,t(x,y,z),z)$
is the dual discriminator function in every member of $\mathcal{S}$.
Hence, discriminator varieties are dual discriminator varieties as
well. In fact, discriminator varieties are exactly the congruence-permutable
dual discriminator varieties \cite[Lemma 2.2]{Pixley_Fried_Dual_Disc}.

A seminal result by Bulman-Fleming, Keimel and Werner says that every
algebra in a discriminator variety is isomorphic to a Boolean product
with trivial or simple factors \cite{KeimelWerner1974,BulmanFlemingWerner1977,Werner1978}.
We show next how this theorem can be derived from the results in this
paper.
\begin{cor}
Every algebra in a discriminator variety is isomorphic to a Boolean
product with trivial or simple factors.
\end{cor}

\begin{proof}
Let $\mathcal{V}$ be a discriminator variety. Let $\mathbf{C}\leq_{\mathrm{isd}}\mathbf{S}\times\mathbf{S}'$
for some $\mathbf{S},\mathbf{S}'\in\mathcal{V}_{\mathrm{S}}$, and
let $\gamma$ and $\gamma'$ be the kernels of the coordinate projections.
Then, as $\gamma$ and $\gamma'$ are permuting maximal congruences
of $\mathbf{C}$, it must be that $\mathbf{C}=\mathbf{S}\times\mathbf{S}'$.
Thus, there are no simple crosses in $\mathcal{V}$, and Corollary~\ref{cor:global dd}
implies that $\maxc(\mathbf{A})$ is a global spectrum for every $\mathbf{A}\in\mathcal{V}$.
Moreover, if $\mathbf{A}\in\mathcal{V}$, then $\bigcap\maxc(\mathbf{A})=\Delta$,
the members of $\maxc(\mathbf{A})$ are pairwise incomparable, and
item (1) of Lemma~\ref{lem:compacidad 1} says that $\maxc(\mathbf{A})\cup\{\nabla\}$
is $\ted$-closed. Proposition~\ref{prop:global incomparable es booleano}
now shows that $\maxc(\mathbf{A})\cup\{\nabla\}$ is a Boolean spectrum
of $\mathbf{A}$. Its corresponding factors are simple, except for
the trivial factor $\mathbf{A}/\nabla$.
\end{proof}

\section{Ado-semilattices: a case study\label{sec:ado-semilattices}}

\subsection{Ado-semilattices}

Let $X$ and $Y$ be sets, and let $\operatorname{Par}(X,Y)$ denote
the set of partial functions from $X$ to $Y$. For $f,g\in\operatorname{Par}(X,Y)$
the \emph{override} of $f$ by $g$ is defined by 
\[
f\vartriangleright g:=f\cup\left(g|_{\dom(g)\setminus\dom(f)}\right).
\]
Thus, $f\vartriangleright g$ agrees with $f$ wherever $f$ is defined,
and otherwise agrees with $g$. We also consider the usual intersection
$f\cap g$ of the graphs of $f$ and $g$.

An \emph{ado-semilattice} is an associative distributive $o$-semilattice
in the sense of C\={\i}rulis \cite{Cirulis2011,Stokes2024}. Ado-semilattices,
considered as algebras in the language $\{\cap,\vartriangleright\}$,
where both operation symbols are binary, form a finitely based variety
which we denote by $\mathcal{A}$. They are precisely the algebras
representable by partial functions, with $\cap$ interpreted as intersection
and $\vartriangleright$ as override \cite[Theorem~4.10]{Stokes2024}.
The override operation has also received considerable attention in
theoretical computer science, where it arises in program-specification
contexts and is closely connected with the if--then--else construct
\cite{JacksonStokes2021,Stokes2024}.

An ado-semilattice $\mathbf{A}$ is \emph{flat} if the order induced
by $\cap$ has a least element $0$ and every element distinct from
$0$ is maximal. The operations of a flat ado-semilattice are given
by 
\[
x\cap y=\begin{cases}
x, & x=y,\\
0, & x\neq y,
\end{cases}\qquad x\vartriangleright y=\begin{cases}
x, & x\neq0,\\
y, & x=0.
\end{cases}
\]
Every flat ado-semilattice is simple, and every ado-semilattice is
a subdirect product of flat ones \cite[Corollary~4.8 and Lemma~4.9]{Stokes2024}.
Consequently, $\mathcal{A}$ is semisimple and its simple members
are precisely the nontrivial flat ado-semilattices. Moreover, the
class $\mathcal{A}_{\mathrm{S}}$ of simple members is almost universal:
subalgebras and ultraproducts of flat algebras are again flat, up
to the trivial algebra.

Finally, $\mathcal{A}$ has the majority term 
\[
M(x,y,z):=\bigl((x\cap y)\vartriangleright(x\cap z)\bigr)\vartriangleright(y\cap z).
\]
Indeed, in every functional representation the three pairwise intersections
are compatible, so the iterated override appearing above coincides
with their union. It follows that $\mathcal{A}$ is congruence distributive.
Therefore, $\mathcal{A}$ is a filtral variety with a majority term.

To characterize the globally indecomposable ado-semilattices we need
the following combinatorial analysis of twofold irredundant subdirect
products.
\begin{lem}
\label{lem:ado g-relation binaria}Let $\mathbf{A}\leq\mathbf{B}$
be nontrivial flat ado-semilattices with a common bottom element $0$.
Suppose that 
\[
\mathbf{G}\leq_{\mathrm{isd}}\mathbf{A}\times\mathbf{B}\text{ and }G\neq A\times B.
\]
Then the following hold.
\begin{enumerate}
\item If $A=B=\{0,1\}$, then 
\[
G=\{\langle0,0\rangle,\langle0,1\rangle,\langle1,1\rangle\}\text{ or }G=\{\langle0,0\rangle,\langle1,0\rangle,\langle1,1\rangle\}.
\]
\item If $A=\{0,1\}\subsetneq B$, then there is $C\subsetneq B\setminus\{0\}$
such that 
\[
G=(\{1\}\times B)\cup\bigl(\{0\}\times(C\cup\{0\})\bigr).
\]
\item If $|A|>2$, then there are $a\in A\setminus\{0\}$ and $b\in B\setminus\{0\}$
such that 
\[
G=(\{a\}\times B)\cup(A\times\{b\})\cup\{\langle0,0\rangle\}.
\]
\end{enumerate}
\end{lem}

\begin{proof}
We begin by establishing several facts that follow from the assumptions
of the lemma. Each of the following statements has a symmetric counterpart
obtained by interchanging $\mathbf{A}$ and $\mathbf{B}$, so we state
and prove only one version.
\begin{enumerate}
\item \label{fact:zero}$\langle0,0\rangle\in G$.

Since $G$ is subdirect, $\langle a,0\rangle\in G$ and $\langle0,b\rangle\in G$
for some $a,b$. Hence $\langle0,0\rangle=\langle a,0\rangle\cap\langle0,b\rangle\in G$.
\end{enumerate}
Define 
\[
A_{0}:=\{a\in A\setminus\{0\}:\langle a,0\rangle\in G\}\quad B_{0}:=\{b\in B\setminus\{0\}:\langle0,b\rangle\in G\}.
\]

\begin{enumerate}[resume]
\item \label{fact:A0xB}$A_{0}\times B\subseteq G$.

Fix $a\in A_{0}$ and $b\in B$. Since $G$ is subdirect, there is
$a'\in A$ such that $\langle a',b\rangle\in G$. Thus, $\langle a,b\rangle=\langle a,0\rangle\vartriangleright\langle a',b\rangle\in G$.
\item \label{fact:large A0 imp 0xB} If $|A_{0}|>1$, then $\{0\}\times B\subseteq G$.

\noindent Suppose $a,a'\in A_{0}$ are distinct and let $b\in B$.
By \ref{fact:A0xB} we have $\langle a,b\rangle,\langle a',b\rangle\in G$,
and hence $\langle0,b\rangle=\langle a,b\rangle\cap\langle a',b\rangle\in G$.
\item \label{fact:A0 a lo sumo uno} If $|B|>2$, then $|A_{0}|\leq1$.

Suppose, towards a contradiction, that $|A_{0}|>1$. By \ref{fact:large A0 imp 0xB},
we have $\{0\}\times B\sub G$, and hence $B_{0}=B\setminus\{0\}$.
The symmetric version of \ref{fact:A0xB} therefore gives $A\times(B\setminus\{0\})\sub G$.
Since $|B_{0}|>1$, the symmetric version of \ref{fact:large A0 imp 0xB}
also gives $A\times\{0\}\sub G$. Thus $G=A\times B$, a contradiction.
\end{enumerate}
Define 
\[
A_{*}:=\{a\in A:\exists b\neq b'\,\langle a,b\rangle,\langle a,b'\rangle\in G\}\quad B_{*}:=\{b\in B:\exists a\neq a'\,\langle a,b\rangle,\langle a',b\rangle\in G\}.
\]

\begin{enumerate}[resume]
\item \label{fact:A*-0 es A0.}$A_{*}\setminus\{0\}=A_{0}$.

Suppose $a\in A_{*}\setminus\{0\}$ and choose $b\neq b'$ such that
$\langle a,b\rangle,\langle a,b'\rangle\in G$. Then $\langle a,0\rangle=\langle a,b\rangle\cap\langle a,b'\rangle\in G$,
and thus $a\in A_{0}$. Conversely, if $a\in A_{0}$, then $\{a\}\times B\sub G$
by \ref{fact:A0xB}. Since $B$ is nontrivial, this says that $a\in A_{*}$.
\item \label{fact:A* no vacio}$A_{*}\neq\emptyset$.

Note that if $G$ is the graph of a function, then $G$ is redundant.
\item \label{fact:A0 singleton}If $|B|>2$, then $\left|A_{0}\right|=1$.

By \ref{fact:A0 a lo sumo uno} it is enough to show that $A_{0}\neq\emptyset$.
Suppose otherwise. Then $0\notin B_{*}$, and as $B_{*}\neq\emptyset$
by \ref{fact:A* no vacio}, there is $b\in B_{*}\setminus\{0\}$.
So by \ref{fact:A*-0 es A0.} we have $b\in B_{0}$, and the symmetric
version of \ref{fact:A0xB} gives $A\times\{b\}\sub G$. Take $b'\in B\setminus\{0,b\}$
and pick $a\in A$ with $\langle a,b'\rangle\in G$. Necessarily $a\neq0$,
since on the contrary $|B_{0}|>1$. Now, $\langle a,0\rangle=\langle a,b\rangle\cap\langle a,b'\rangle\in G$,
a contradiction.
\end{enumerate}
We now apply these facts to prove the assertions in the lemma.

(1). Suppose $A=B=\{0,1\}$. By \ref{fact:A* no vacio}, $A_{*}\neq\emptyset$.
If $1\in A_{*}$, then $1\in A_{0}$ by \ref{fact:A*-0 es A0.}, and
hence $\{1\}\times B\sub G$ by \ref{fact:A0xB}. Together with \ref{fact:zero}
and the properness of $G$, this yields 
\[
G=\{\langle0,0\rangle,\langle1,0\rangle,\langle1,1\rangle\}.
\]
Otherwise, $0\in A_{*}$, so $\langle0,1\rangle\in G$ and therefore
$1\in B_{0}$. The symmetric version of \ref{fact:A0xB} now gives
$A\times\{1\}\sub G$, and properness yields 
\[
G=\{\langle0,0\rangle,\langle0,1\rangle,\langle1,1\rangle\}.
\]

(2). Suppose $A=\{0,1\}\subsetneq B$. Note that \ref{fact:A0 singleton}
gives $A_{0}=\{1\}$, and thus $\{1\}\times B\sub G$ by \ref{fact:A0xB}.
Setting $C:=B_{0}$, we obtain 
\[
G=(\{1\}\times B)\cup(\{0\}\times(C\cup\{0\})).
\]
Finally, $C\subsetneq B\setminus\{0\}$ because $G\neq A\times B$.

(3). Assume $\left|A\right|>2$. By \ref{fact:A0 a lo sumo uno} and
its symmetric version, the sets $A_{0}$ and $B_{0}$ are singletons;
say $A_{0}=\{a\}$ and $B_{0}=\{b\}$. We claim that 
\[
G=(\{a\}\times B)\cup(A\times\{b\})\cup\{\langle0,0\rangle\}.
\]
The right-to-left inclusion follows from \ref{fact:zero} and \ref{fact:A0xB}.
Conversely, let $\langle a',b'\rangle\in G$. If $a'=a$ or $b'=b$,
then $\langle a',b'\rangle$ belongs to the right-hand side. Otherwise,
flatness gives $\langle0,b'\rangle,\langle a',0\rangle\in G$. So,
as $a'\notin A_{0}$ and $b'\notin B_{0}$, we have $\langle a',b'\rangle=\langle0,0\rangle$,
as required.

We call an ado-semilattice a \emph{one-sided ado-wing} if it is isomorphic
to one of the algebras described in items (1) and (2) of Lemma~\ref{lem:ado g-relation binaria},
and a \emph{two-sided ado-wing} if it is isomorphic to an algebra
described in item (3). An \emph{ado-wing} is a one- or two-sided ado-wing.
\end{proof}

\begin{figure}[H]
\centering
\begin{tabular}{>{\centering}p{0.48\textwidth}>{\centering}p{0.48\textwidth}}
\includegraphics[width=1\linewidth]{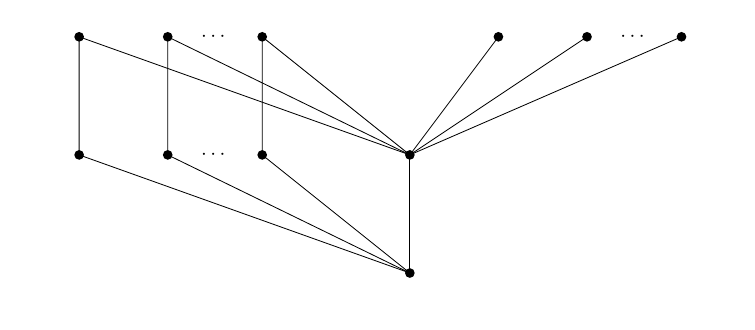} & \includegraphics[width=1\linewidth]{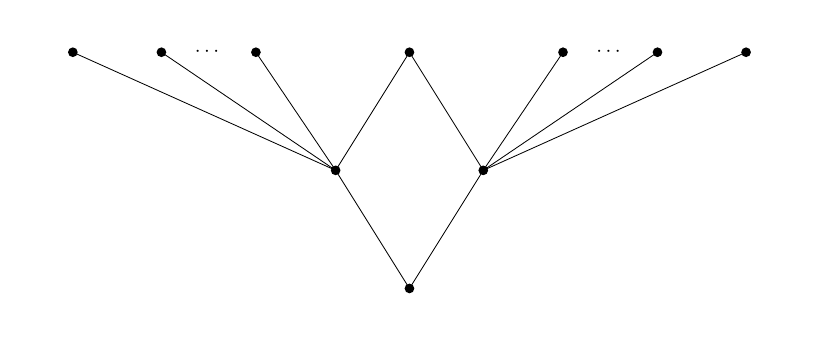}\tabularnewline
\end{tabular}

\caption{One- and two-sided wings.}

\end{figure}

\begin{lem}
\label{lem:ado-wings}An ado-semilattice is GI if and only if its
semilattice reduct is a wing.
\end{lem}

We are now ready to obtain a neat global representation result for
$\mathcal{A}$.
\begin{cor}
The variety $\mathcal{A}$ is GI-spectral and 
\[
\mathcal{A}_{\mathrm{GI}}=\mathcal{A}_{\mathrm{S}}\cup\{\mathbf{F}\in\mathcal{A}:\mathbf{F}\text{ is and ado-wing}\}.
\]
\end{cor}

\begin{proof}
By Corollary~\ref{cor:filtral implica RGI spectral}, Remark~\ref{rem:g-algebras de ancho 2},
and Corollary~\ref{cor:toda g-algebra es RGI}, the nonsimple globally
indecomposable members of $\mathcal{A}$ are precisely the proper
irredundant subdirect products of two simple algebras. Since the simple
ado-semilattices are precisely the nontrivial flat ones, the result
follows from Lemma~\ref{lem:ado g-relation binaria}.
\end{proof}

We can also characterize congruence-permutable ado-semilattices.
\begin{cor}
For ado-semilattice $\mathbf{A}$ the following are equivalent:
\begin{enumerate}
\item $\mathbf{A}$ is congruence-permutable
\item $\mathbf{A}$ has no ado-wings as quotients.
\end{enumerate}
\end{cor}

\begin{proof}
This follows from Theorem~\ref{thm:tipo Nachbin para superchela}
together with Lemma~\ref{lem:ado-wings}.
\end{proof}

\section*{Declaration of generative AI and AI-assisted technologies in the
manuscript preparation process}

During the preparation of this work, the author used OpenAI's ChatGPT
for two specific purposes. First, ChatGPT suggested the initial idea
underlying the counterexample presented in Section\textasciitilde\ref{subsec:contraejemplo},
in response to the problem and constraints formulated by the author.
The author then substantially modified the proposed construction and
developed its final form and proof. Second, ChatGPT was used for language
editing and proofreading of portions of the manuscript. All AI-generated
suggestions were critically reviewed by the author, who independently
verified the mathematical content and takes full responsibility for
the article.

\bibliographystyle{plain}
\bibliography{global}

\end{document}